\documentclass{amsart}    
\usepackage{amsmath} 
\usepackage{paralist}
\usepackage{graphics} 
\usepackage{epsfig} 
\usepackage{graphicx}  \usepackage{epstopdf} 
\usepackage[colorlinks=true]{hyperref}
 \usepackage[utf8]{inputenc}
\hypersetup{urlcolor=blue, citecolor=red}

\newtheorem{Proposition}{Proposition}[section]

\newtheorem{Lemme}{Lemma}[section]
\newtheorem{Theoreme}{Theorem}[section]

\newtheorem{Remarque}{Remark}
\newcommand{\sign}{\text{sign}}
\def \tu{\tilde{u}}
\def \bu{{\bf u}}
\def \tv{\tilde{v}}
\def \bK{{\bf K}}
\def \R{\mathbb{R}}
\def \ds{\displaystyle} 
 
\usepackage{amssymb} 

\title[\bf Nonlocal nonlinear model] %
{Remarks on the Spatial Asymptotic Behavior of Solutions to a 1D Model of Equatorial Oceanic Flows} 
\author[ Manuel Fernando Cortez  and Oscar Jarr\'in]{}


\email{oscar.jarrin@udla.edu.ec}
\email{manuel.cortez@epn.edu.ec} 

\thanks{$^*$Corresponding author:  Oscar Jarr\'in}

\begin{document}
	\maketitle
	\begin{center}
	    \begin{minipage}{5cm}
 	\centerline{\scshape Manuel Fernando Cortez}
	\medskip
	{\footnotesize
		\centerline{Departamento de Matem\'aticas}
		\centerline{Escuela Politécnica Nacional} 
		\centerline{Ladr\'on de Guevera E11-253, Quito, Ecuador} 
	}
        \end{minipage}\hspace{2cm}
        \begin{minipage}{5cm}
           	\centerline{\scshape Oscar Jarr\'in$^*$}
	\medskip
	{\footnotesize
		\centerline{Escuela de Ciencias Físicas y Matemáticas}
		\centerline{Universidad de Las Américas}
		\centerline{V\'ia a Nay\'on, C.P.170124, Quito, Ecuador}
	} 
        \end{minipage}
	\end{center}
	
\bigskip
\begin{abstract} We consider a new nonlocal and nonlinear one-dimensional evolution model arising in the study of oceanic flows in equatorial regions, recently derived in [A. Constantin and L. Molinet, Global Existence and Finite-Time Blow-Up for a Nonlinear Nonlocal Evolution Equation, Commun. Math. Phys. 402 (2023), 3233–3252].
	
We investigate the spatial asymptotic behavior of its solutions. In particular, we observe the influence of the Coriolis effect, which, even for rapidly decaying initial data, yields solutions that decay at the rate $1/|x|$.
Thereafter, we shed light on the optimality of this decay rate.
\end{abstract}

\section{Introduction}
\subsection*{Introduction to the model}  Very recently, in \cite{Constantin-Molinet}, P. Constantin and L. Molinet derived a new nonlocal and nonlinear one-dimensional evolution model arising in the study of \emph{oceanic flows in equatorial regions}. In what follows, we provide a brief outline of the derivation of this model. For a more detailed exposition, see Section 2 of \cite{Constantin-Molinet}. 

\medskip

Within the framework of the physical coordinates -  where the $x$-axis points horizontally due east, the $y$-axis points horizontally due north, and  the $z$-axis pointing vertically - the starting point are the Navier–Stokes system governing equatorial ocean dynamics on the $f$-plane, \emph{i.e.},  where the Coriolis parameter is approximated as a constant:
\begin{equation*}
\begin{cases}\vspace{2mm}
\partial_t u + u\, \partial_x u + v\, \partial_y u + w\, \partial_z u - \beta w =-\frac{1}{\rho}\partial_x P+\mu_1 (\partial^2_x u + \partial^2_y u)+ \mu_2 \partial^2_z u, \\ \vspace{2mm}
\partial_t v + u\, \partial_x v + v\, \partial_y v + w\, \partial_z v = -\frac{1}{\rho}\partial_y P+\mu_1 (\partial^2_x u + \partial^2_y u)+ \mu_2 \partial^2_z u, \\ \vspace{2mm}
\partial_t w + u\, \partial_x w + v\, \partial_y w + w\, \partial_z w + \beta w =-\frac{1}{\rho}\partial_w P-g+\mu_1 (\partial^2_x w + \partial^2_y w)+ \mu_2 \partial^2_z w, \\
\partial_x u + \partial_y v + \partial_z w=0.
\end{cases}
\end{equation*}
Here,  $u=u(t,x,y,z)$, $v=v(t,x,y,z)$ and $w=w(t,x,y,z)$ denote the components of the fluid velocity in the directions of azimuth, latitude, and elevation, respectively, and $P=P(t,x,y,z)$ is the pressure. Additionally, $\rho>0$ is the constant water density, $g>0$ is the Earth's gravitational acceleration at the surface, $\mu_1, \mu_2>0$ are the horizontal and vertical viscosity coefficients,  and $\beta \neq 0$ is a parameter characterizing the Coriolis effect.  Finally, the equation $\partial_x u + \partial_y v + \partial_z w=0$ describes the mass conservation. 

\medspace

Exploiting the confinement of \emph{equatorial flows}, the meridional velocity component $v$,  as well as the  $y$-axis can be neglected, yielding the system for $u=u(t,x,z)$ and $w=w(t,x,z)$:
 \begin{equation*}
 \begin{cases}\vspace{2mm}
 \partial_t u + u\, \partial_x u  + w\, \partial_z u - \beta w =-\frac{1}{\rho}\partial_x P+\mu_1 \partial^2_x u + \mu_2 \partial^2_z u, \\ \vspace{2mm}
 \partial_t w + u\, \partial_x w  + w\, \partial_z w + \beta w =-\frac{1}{\rho}\partial_w P-g+\mu_1 \partial^2_x w+ \mu_2 \partial^2_z w, \\
 \partial_x u  + \partial_z w=0.
 \end{cases}
 \end{equation*}
 
Thereafter, by restricting the motion to a fixed depth $z=z_0$ and using the incompressibility constraint, a harmonic stream function is introduced to rigorously deduce that the vertical velocity $w$ is related to the horizontal velocity $u$ by 
\[ w(t,x,z_0)= - \mathcal{H}u(t,x,z_0),\]
where $\mathcal{H}$ denotes the Hilbert transform (for an explicit definition, see  expression (\ref{Hilbert}) below).  Additionally, it is deduced that the horizontal pressure gradient acts as an external forcing term
\[ - \frac{1}{\rho} \partial_x P(t,x,z_0):= f(t,x). \]
 Finally, by setting a unified viscosity constant 
\[ \mu:= \mu_1- \mu_2>0, \]
the authors of \cite{Constantin-Molinet} show that the previous system reduces to the following one-dimensional nonlocal nonlinear model:
\begin{equation*}
  \partial_t u + u\, \partial_x u +  (\mathcal{H}u)\, \partial_x (\mathcal{H}u) + \beta \,\mathcal{H}u - \mu\, \partial^2_x u = f, \qquad \mu>0, \quad \beta \neq 0, 
\end{equation*}
which describes the evolution of the \emph{horizontal velocity}  $u(t,x)$ in the context of \emph{equatorial oceanic flows}. 

\subsection*{Previous theoretical results}  The initial value problem associated with this model is also studied in \cite{Constantin-Molinet} under the periodic spatial condition $x \in \mathbb{T}:= \R \setminus \mathbb{Z}$.  Specifically, for $s\in \R$, we denote by $H^s_0(\mathbb{T})$ the closed subspace of zero-mean functions in the Sobolev space $H^s(\mathbb{T})$. It is then  shown that the model above is locally well-posed in the space $H^s_0(\mathbb{T})$ for $s>-\frac{1}{2}$, provided that the initial datum and  the external source term  satisfy 
\[ u_0 \in H^s_0(\mathbb{T}), \qquad f \in L^\infty\big([0, +\infty[, H^{s'}_{0}(\mathbb{T})\big), \quad  s'\geq s.\]
Moreover, taking $f =0$  and using similar arguments as in  \cite{Bekiranov,Dix}, it follows that $H^{-\frac{1}{2}}_0(\mathbb{T})$ is the critical Sobolev space for the well-posedness of this model. 

\medskip

As noticed in \cite{Constantin-Molinet}, all suitable assumptions can always be imposed on the external source term $f$. Consequently, with only a minor loss of generality, one may assume 
\[ f=0,\]
 in order to carry out a qualitative study of the properties of this model, which are mainly governed by its linear and nonlinear terms. One such property concerns the global-in-time existence of solutions.

\medskip

In full generality, smooth solutions formally satisfy the following energy estimate:
\begin{equation*}
 \frac{d}{dt}\| u(t,\cdot)\|^2_{L^2} \lesssim - \int_{\mathbb{T}} (\mathcal{H}u)\, \partial_x (\mathcal{H}u)\, u \, dx-\| u(t,\cdot)\|^2_{L^2}, 
\end{equation*}
where the nonlinear effects of the nonlocal term  $(\mathcal{H}u) \, \partial_x (\mathcal{H}u)$ ultimately prevent the derivation of suitable energy estimates that would yield global-in-time existence of solutions through standard Grönwall-type arguments.

\medskip

In this context, for $s \geq 0$,  it is shown in \cite{Constantin-Molinet}  that small  initial data $u_0 \in H^s_0(\mathbb{T})$ satisying $\| u_0 \|_{L^2} \leq \frac{\mu}{2}$, yield global-in-time  solutions $u \in \mathcal{C}([0,+\infty[, H^s_0(\mathbb{T}))$ that satisfy the uniform bound $\| u(t,\cdot)\|_{L^2} \leq \frac{\mu}{4}$, for all $t>0$.

\medskip

The long-time dynamics of solutions arising from large initial data, \emph{i.e.}, when $\| u_0 \|_{L^2}> \frac{\mu}{2}$, appear to be considerably more complex. In fact,  by exploiting the Fourier mode decomposition, 
the last sections of \cite{Constantin-Molinet} show, on the one hand, that certain large initial datum supported on a finite number of Fourier modes yield a global-in-time solution $u(t,x)$ to the model.  On the other hand,  this same  type of initial data can lead to blow-up in finite time of the corresponding solution.

\subsection*{New theoretical results} The main objective of this note is to continue the qualitative study of the model introduced above. Specifically, we investigate the spatial asymptotic behavior of its solutions, assuming from now on that the spatial variable $x$  belongs to the entire real line $\R$. 

\medskip

The spatial asymptotics of solutions to models in fluid motion are not only of mathematical interest but also of physical relevance. Intuitively, as mentioned in \cite{Brandolese-Vigneron}, if at the beginning of the evolution problem the fluid is at rest outside a bounded region (\emph{i.e}., the initial data are compactly supported), we would like to determine the rate at which the fluid particles move away from the initial region once the motion begins. 

\medskip

Mathematically, such questions are addressed by studying the \emph{pointwise behavior} of solutions to these models as  $|x|\to +\infty$. For instance, in the case of the \emph{three-dimensional Navier–Stokes} equations on the whole space, it was shown in \cite{Dobrokhotov} that well-prepared initial data lead to an instantaneous spreading of the corresponding solutions, which cannot decay at infinity faster than the rate  $1/|x|^4$. See also \cite{Bjorland,Brandolese-Meyer, Brandolese-Vigneron, Brandolese2} for further references on the spatial pointwise behavior of the Navier–Stokes equations

\medskip

Concerning some one-dimensional models, this question has been studied in \cite{Cortez-Jarrin-1, Cortez-Jarrin-2} for certain \emph{dissipative modifications} of the \emph{Korteweg–De Vries}  and \emph{Benjamin–Ono} equations, which arise in the study of viscous stratified fluids. In these cases, it was shown that well-prepared initial data lead to solutions with an optimal pointwise decay rate of $1/|x|^2$  as $|x|\to +\infty$. Recently, in \cite{Cortez-Jarrin3}, this study is done for a \emph{generalized dispersive-dissipative Kuramoto-type equation}, leading pointwise decaying rates of solutions which are essentially governing by the fractional power of the Laplacian operator appearing in this equation.

\medskip

Motivated by the aforementioned studies, we now turn our attention to the initial value problem associated with the nonlocal nonlinear model introduced above
\begin{equation}\label{Main-Equation}
\begin{cases}\vspace{2mm}
\partial_t u + u\, \partial_x u +  (\mathcal{H}u)\, \partial_x (\mathcal{H}u) + \beta \,\mathcal{H}u - \mu\, \partial^2_x u = 0, \qquad \beta \neq 0, \quad \mu >0,\\
u(0,\cdot)=u_0,
\end{cases}
\end{equation}
where, for a time $T>0$, the function $u:[0,T]\times \R \to \R$ denotes the solution,  and $u_0: \R \to \R$ an initial datum.  Additionally, the operator $\mathcal{H}$ is the Hilbert's transform, which can be defined in the Fourier level by the symbol 
\begin{equation}\label{Hilbert}
\widehat{\mathcal{H}\varphi}(\xi)=-i\,\sign(\xi)\widehat{\varphi}(\xi), \qquad 
\sign(\xi)=\,\begin{cases}
-1, \quad \xi <0, \\
0, \quad \xi=0, \\
1, \quad \xi>0,
\end{cases} 
\qquad \varphi \in \mathcal{S}(\R). 
\end{equation}

Note that, with a minor loss of generality, the external source term is taken to be zero. Here, we aim to understand how the linear and nonlinear terms in this equation govern the asymptotic spatial behavior of the solutions. 

\medskip

We will consider the following hypothesis on the initial data:
\begin{equation}\label{Def-Data-0}
u_0 \in H^s(\R), \qquad s > \frac{3}{2}.
\end{equation}
This constraint on the regularity parameter $s$ is essentially technical, introduced to handle the nonlinear terms $u\,\partial_x u$ and  $(\mathcal{H}u)\, \partial_x (\mathcal{H}u)$ in equation (\ref{Main-Equation}).

\medskip

Furthermore,  for a constant $C_0>0$ and a fixed parameter $\gamma>0$, we assume the prescribed pointwise  decaying rate 
	\begin{equation}\label{Def-Data}
|u_0(x)|\leq \frac{C_0}{(1+|x|)^{1+\gamma}}, \qquad \text{for any}\quad  x \in \R.
\end{equation}
Here, the parameter  $\gamma$ essentially measures the decay rate assumed for the initial data. Nevertheless, due to certain technical constraints, a faster decay rate, given by the exponent $1+\gamma$, 
is required.

\medskip

Then, our main result is stated as follows:
\begin{Theoreme}\label{Th-Main} Assume  (\ref{Def-Data-0}) and (\ref{Def-Data}).  There exists a time $T_0=T_0( u_0)>0$ and a unique solution $u \in \mathcal{C}\big([0,T_0], H^s(\R) \big)$ of equation (\ref{Main-Equation}). Additionally, for any time $0<t\leq T_0$ and $x\in \R$, the arising solution $u(t,x)$ of equation (\ref{Main-Equation})  satisfies the pointwise estimates:
	\begin{equation}\label{Decay-Solution}
	|u(t,x)|+|\mathcal{H}u(t,x)| \leq \frac{C_1}{t^{\frac{1}{2}}(1+|x|)^{\min(1,\gamma)}},    
	\end{equation}
where the constant $C_1 = C_1(\beta, \gamma, \mu, T_0, C_0, u) > 0$ is independent of the time variable $t$ and the spatial variable $x$.
\end{Theoreme} 

 Estimate (\ref{Decay-Solution}) exhibits a pointwise spatial decay rate for the solution $u(t,x)$, representing the horizontal velocity, and for its Hilbert transform $\mathcal{H}u(t,x)$,  representing the vertical velocity in the context of the equatorial oceanic flows introduced above.

\medskip

When comparing estimates (\ref{Def-Data}) and (\ref{Decay-Solution}), it is interesting to observe an instantaneous loss of persistence with respect to the prescribed decay rate assumed for the initial data. Specifically, as $|x|\to +\infty$ we have 
\begin{equation*}
|u_0(x)|\lesssim \frac{1}{|x|^{1+\gamma}}    \quad  \text{at}\quad t=0, \quad \text{and}\quad  |u(t,x)| \lesssim 
\begin{cases}\vspace{3mm}
 \ds{\frac{1}{|x|^\gamma}}, &   0<\gamma<1, \\
 \ds{\frac{1}{|x|}}, &\gamma>1,
\end{cases}  \qquad \text{for}\quad  0<t\leq T_0.
\end{equation*}
We thus conclude that for small values of $\gamma$ (that is, when $0<\gamma<1$) the decay rate of the initial data still influences the spatial asymptotics of the solution, whereas for large values of $\gamma$ (when $\gamma>1$),  the solution decays at infinity like  $1/|x|$. 

\medskip

In order to briefly explain this phenomenon, we mention that, roughly speaking, estimate (\ref{Decay-Solution}) is obtained from the mild formulation of the solution $u(t,x)$, which involves a convolution kernel $K(t,x)$. This kernel,  explicitly given in expression (\ref{Kernel}) below,
depends strongly on   the Coriolis parameter $\beta$.  In the physically relevant case $\beta \neq 0$, we show that $\ds{K(t,x) \sim 1/|x|}$ for any $t>0$ and  sufficiently large $|x|$, which leads to the observed decay rate of the solutions when $\gamma >1$. 

\medskip

It is therefore interesting to observe the influence of the Coriolis parameter $\beta$ on the spatial asymptotics of solutions to equation (\ref{Main-Equation}). In the case $\beta = 0$, expression (\ref{Kernel}) shows that $K(t,x)$ is the well-known heat kernel, which decays rapidly in the spatial variable. Nevertheless, this case is not physically relevant in the context of this  equation.  

\medskip

Assuming fast-decaying data, which satisfy (\ref{Def-Data}) for $\gamma > 1$, in our next result we shed some light on the optimality of the decay rate $1/|x|$ of solutions.

\begin{Proposition}\label{Prop-Main-1} Under the same hypotheses (\ref{Def-Data-0}) and (\ref{Def-Data}),  assume that $\gamma>1$ and that the initial datum also satisfies
	\begin{equation}\label{Mass-Initial-Datum}
	\int_{\R} u_0(x)\,dx := M(u_0) \neq 0.
	\end{equation}
	Let $u(t,x)$ be the associated solution of equation (\ref{Main-Equation}) obtained in Theorem \ref{Th-Main}. Then, \emph{one of the following statements does not hold}:
	\begin{itemize}
		\item  There exist $\varepsilon > 0$ and $M_\varepsilon > 0$ such that, for any  $0<t\leq T_0$ and any $|x|>M_\varepsilon$, the solution satisfies the pointwise decay estimate:
		\begin{equation}\label{Decay-Solution-Fast}
		|u(t,x)| + |\mathcal{H}u(t,x)| \leq \frac{C_2}{t^{\frac{1}{2}}\,|x|^{1+\varepsilon}}, 
		\end{equation}
		for some generic constant $C_2 > 0$, independent of $t$ and $x$, but possibly depending on the parameters $\beta, \gamma, \varepsilon, \mu, T_0, C_0, M(u_0)$, as well as on certain norms of the solution $u$.  
		
		\medskip
		
		\item The solution satisfies the mean decay property
		\begin{equation}\label{Decay-Solution-Mean-Fast}
		u, \mathcal{H}u \in \mathcal{C}\Big([0,T_0], L^2\big(\R, |x|\,dx\big)\Big).
		\end{equation}
	\end{itemize}
\end{Proposition}	

In order to explain in what sense this result suggests the optimality of the decay $1/|x|$, we provide a brief explanation of the main idea of the proof. First, we discuss the meaning of (\ref{Decay-Solution-Fast}) and (\ref{Decay-Solution-Mean-Fast}). Note that estimate (\ref{Decay-Solution-Fast}) states that the solution $u(t,x)$ of equation (\ref{Main-Equation}) decays at infinity faster than $1/|x|$, while equation (\ref{Decay-Solution-Mean-Fast}) shows that, for any time $0 \le t \le T_0$, the solution $u(t,x)$ satisfies $|x|u(t,\cdot) \in L^2(\R)$. Such weighted Lebesgue spaces have been used in previous works concerning spatial decay in one-dimensional dispersive models \cite{Fonseca,Munoz,Nahas}. Specifically, the fact that $|x|u(t,x)$ is an $L^2$-function implies that, on average, $u(t,x)$ must decay at infinity faster than $1/|x|$. 

\medskip

In this context, assuming that both (\ref{Decay-Solution-Fast}) and (\ref{Decay-Solution-Mean-Fast}) hold, a contradiction arises. Indeed, under these assumptions, and assuming in addition the technical nonzero-mean condition on the initial data (\ref{Mass-Initial-Datum}), we are able to show that, for any fixed time $0 < t \le T_0$, the solution exhibits the following sharp asymptotic behavior:
\[ u(t,x)=\frac{1}{x} \Phi\big(M(u_0),t,u\big)+ o(t)\left(\frac{1}{|x|}\right), \qquad  |x|\to+\infty, \]
where the function $\Phi\big(M(u_0), t, u\big)$, explicitly defined in equation (\ref{Phi}) below, does not depend on the spatial variable $x$, and the expression $o(t)(1/|x|)$ satisfies 
\[  \lim_{|x|\to+\infty} \frac{o(t)(1/|x|)}{1/|x|}=0. \]
Ultimately, this asymptotic profile yields a contradiction with assumptions (\ref{Decay-Solution-Fast}) and (\ref{Decay-Solution-Mean-Fast}), respectively. 

\medskip

Returning to Theorem \ref{Main-Equation}, note that, as a by-product, we also prove the local well-posedness of equation (\ref{Main-Equation}) in the space $H^s(\R)$ for $s > \frac{3}{2}$. In this sense, we partially extend to the whole real line $\R$ the result obtained in \cite[Proposition 3.3]{Constantin-Molinet} for the torus $\mathbb{T}$. Additionally, following similar arguments as in \cite[Proposition 3.4]{Constantin-Molinet}, the local well-posedness in $H^s$-spaces can be extended to $s > -\frac{1}{2}$. Moreover, the question of local well-posedness in the critical space $H^{-\frac{1}{2}}$, in both the periodic and non-periodic cases, appears to be far from obvious.

\medskip

On the other hand, the global-in-time existence of $H^s$-solutions to equation (\ref{Main-Equation}) is an interesting question, since the arguments used in the periodic setting are no longer valid on the whole real line $\R$. In fact, the existence of small $L^2$-solutions belonging to the space $\mathcal{C}([0,+\infty), H^s_0(\mathbb{T}))$ for $s \ge 0$, proved in \cite[Proposition 4.1]{Constantin-Molinet}, relies crucially on Poincaré’s inequality for zero-mean functions. Moreover, the finite-time blow-up of solutions for certain well-prepared large initial data, established in \cite[Proposition 6.1]{Constantin-Molinet}, also exploits tools specific to the periodic setting, in particular, the Fourier-mode decomposition of periodic, zero-mean, square-integrable functions

\medskip

Consequently, on the whole real line $\R$, the long-time dynamics of $H^s$-solutions to equation (\ref{Main-Equation}) remains, to the best of our knowledge, a challenging open problem. In future work, we aim to gain a deeper understanding of this issue. Nevertheless, by performing suitable energy estimates, we are able to establish the following blow-up criterion.

\begin{Proposition}\label{Prop-Main-2} Under the same hypothesis as in (\ref{Def-Data-0}),  let $u \in \mathcal{C}\big([0,T_0], H^s(\R) \big)$ be the solution to equation (\ref{Main-Equation}) obtained in Theorem \ref{Th-Main}. Then the following statement holds: for some time 
$T_0<T_*<+\infty$, we have
		\begin{equation}\label{Blow-up-Criterion}
		\lim_{t \to T_{*}}\| u(t,\cdot)\|_{H^s}=+\infty  \quad \text{if and only if} \quad \int_{0}^{T_{*}} \| \partial_x u(t,\cdot)\|_{L^\infty}\, dt =+\infty. 
		\end{equation}
%
%
\end{Proposition}	

From this blow-up criterion and well-known properties of the Hilbert transform, for a given time $T_0 < T_*$, assuming that 
\[  \int_{0}^{T_{*}} \| \partial_x u(t,\cdot)\|_{L^\infty}\, dt <+\infty, \] 
it follows that the solution $u(t,x)$ of equation (\ref{Main-Equation}) satisfies  
\[ u, \, \mathcal{H} u  \in \mathcal{C}\big([0,T_*], H^s(\R)\big). \]
It is then natural to ask whether this information on $u$ allows the pointwise spatial decay estimate (\ref{Decay-Solution}) to be extended to all times $0<t\leq T_*$.  Nevertheless, to the best of our knowledge, this does not seem to be possible in full generality. In this setting, in order to extend the pointwise estimate (\ref{Decay-Solution}) to further times, we require an additional technical assumption on the solution $u(t,x)$. 
 
\begin{Proposition}\label{Prop-Main-3} The pointwise estimates (\ref{Decay-Solution}) for the solution $u(t,x)$ remain valid for all times $0<t<T_{*}$, with  $T_0 < T_{*}<+\infty$, provided that  
			\begin{equation}\label{Condition-L1-Solution}
		u, \, \mathcal{H}u \in \mathcal{C}\left([0,T_*], H^s \cap L^1(\R)\right). 
			\end{equation}
\end{Proposition}

Even when the initial data belong to  $L^1(\R)$, it is not well understood whether this property is preserved by the associated solution of equation (\ref{Main-Equation}); see Remark \ref{Rmk-2} for more details on this point. Consequently, in this result we provide a partial answer to the question raised above.

\medskip

{\bf Organization of the article}. The rest of the paper is organized as follows. Our proof relies strongly on the mild formulation of equation (\ref{Equation-Mild}), specifically on the convolution kernel arising from the linear terms. In Section \ref{Sec:Kernel}, we introduce this kernel and derive several useful estimates. Section \ref{Sec:Spatial-Decay} is devoted to the proof of Theorem \ref{Th-Main} and Proposition \ref{Prop-Main-1}. Finally, in Section \ref{Sec:Global-in-time}, we provide the proof of Propositions \ref{Prop-Main-2} and \ref{Prop-Main-3}. 

\medskip

{\bf Notation}. Throughout this note, $\widehat{\varphi}(\xi)$ denotes the Fourier transform of $\varphi$. Additionally, $C>0$ denotes a generic constant, depending on the parameters of the model, which may change from one line to the next.

\section{Kernel estimates}\label{Sec:Kernel}
We introduce the kernel $K(t,x)$ as the solution of the linear problem associated to equation (\ref{Main-Equation}):
\begin{equation*}
    \begin{cases}\vspace{1mm}
    \partial_t  K + \beta\, \mathcal{H} K - \mu\,\partial^2_x K=0, \qquad \beta\neq 0, \quad \mu>0,\\
    K(t,0)=\delta_0,
    \end{cases}
\end{equation*}
where $\delta_0$ denotes the Dirac mass at the origin. A direct computation yields the explicit expression for $K(t,x)$:
\begin{equation}\label{Kernel}
    K(t,x)=\int_{-\infty}^{+\infty} e^{2\pi i x \xi} e^{-\mu \xi^2 t + i \beta \sign(\xi)t} d\xi. 
\end{equation}
Using this representation, and following some ideas in \cite{Cortez-Jarrin-1,Cortez-Jarrin-2,Cortez-Jarrin3},  we derive some useful pointwise estimates for $K(t,x)$ and $\mathcal{H}K(t,x)$.
\begin{Proposition}\label{Prop-Kernel-Estim} There exists a constant $C=C(\beta,\mu)>0$,  which depends on the parameters $\beta\neq 0$ and $\mu>0$, such that for any $t>0$ and $x\in \R$ it holds:
\begin{equation}\label{Estim-Ker-1}
    | K(t,x)| + |\mathcal{H}K(t,x)| \leq C\,\frac{\eta(t)}{t^{\frac{1}{2}}}\, \frac{1}{1+|x|}, \qquad \eta(t):=1+t^{\frac{1}{2}}+t,
\end{equation}
where $\mathcal{H}$ is the Hilbert's transform defined in expression (\ref{Hilbert}). 
\end{Proposition}
\begin{proof}
We begin by considering the function $K(t,x)$. Using  expression (\ref{Kernel}), the identity $\partial_\xi e^{2\pi i x\xi}= 2\pi i x \, e^{2\pi i x\xi}$, and integrating by parts, we obtain
\begin{equation}\label{Iden-Ker-1}
    \begin{split}
 K(t,x)=&\, \int_{-\infty}^{0} e^{2\pi i x\xi} e^{-\mu \xi^2 t - i \beta t}d \xi+  \int_{0}^{+\infty} e^{2\pi i x\xi} e^{-\mu \xi^2 t + i \beta t}d \xi\\
 =&\, \frac{1}{2\pi i x}\int_{-\infty}^{0} \partial_\xi (e^{2\pi i x\xi}) e^{-\mu \xi^2 t - i \beta t}d \xi+ \frac{1}{2\pi i x} \int_{0}^{+\infty} \partial_\xi(e^{2\pi i x\xi}) e^{-\mu \xi^2 t + i \beta t}d \xi\\
 =&\, - \frac{1}{2\pi i x}\int_{-\infty}^{0} e^{2\pi i x\xi} \partial_\xi (e^{-\mu \xi^2 t - i \beta t})d \xi+ \frac{1}{2\pi i x} \left( \left. e^{2\pi i x\xi}e^{-\mu \xi^2 t - i \beta t}\right|^{0}_{-\infty}\right) \\
 &\,- \frac{1}{2\pi i x}\int_{0}^{+\infty} e^{2\pi i x\xi} \partial_\xi (e^{-\mu \xi^2 t + i \beta t})d \xi+\frac{1}{2\pi i x}\left( \left. e^{2\pi i x\xi}e^{-\mu \xi^2 t + i \beta t}\right|^{+\infty}_{0}\right)\\
 =&\,\frac{e^{-i \beta t}-e^{i\beta t}}{2\pi i x}- \underbrace{\frac{1}{2\pi i x}\left(\int_{-\infty}^{0} e^{2\pi i x\xi} \partial_\xi (e^{-\mu \xi^2 t - i \beta t})d \xi+\int_{0}^{+\infty} e^{2\pi i x\xi} \partial_\xi (e^{-\mu \xi^2 t + i \beta t})d \xi \right)}_{I(t,x)}.
    \end{split}
\end{equation}

Here, for any $x\neq 0$, by direct calculation we obtain 
\begin{equation*}
  \frac{e^{-i \beta t}-e^{i\beta t}}{2\pi i x}= -\frac{\sin(\beta t)}{\pi x}.
\end{equation*}
Additionally, for any $x\neq 0$, the second term is estimated as follows:
\begin{equation}\label{Estim-I}
 I(t,x) \leq C\, \frac{t^{\frac{1}{2}}}{x^2}. 
\end{equation}
In fact, following similar computations as above and using the rapid decay properties of the heat kernel  $e^{-\mu \xi^2 t}$, we write
\begin{equation*}
    \begin{split}
I(t,x) =&\, \frac{1}{(2\pi i x)^2} \left(-\int_{-\infty}^{0} e^{2\pi i x\xi} \partial^2_\xi (e^{-\mu \xi^2 t - i \beta t})d \xi + \left. e^{2\pi i x \xi}e^{-\mu \xi^2 t - i \beta t}(-2\mu \xi t)\right|^{0}_{-\infty} \right.\\
    &\,\left. -\int_{0}^{+\infty} e^{2\pi i x\xi} \partial^2_\xi (e^{-\mu \xi^2 t + i \beta t})d \xi +  \left. e^{2\pi i x \xi}e^{-\mu \xi^2 t + i \beta t}(-2\mu \xi t)\right|^{+\infty}_{0}\right)\\
    =&\, \frac{1}{4\pi^2 x^2} \left( \int_{-\infty}^{0} e^{2\pi i x\xi} e^{-\mu \xi^2 t - i\beta t} (4\mu^2 \xi^2 t^2-2\mu t) d\xi+  \int_{0}^{+\infty} e^{2\pi i x\xi} e^{-\mu \xi^2 t + i\beta t} (4\mu^2 \xi^2 t^2-2\mu t) d\xi \right) \\
    \leq &\, \frac{1}{4\pi^2 x^2} \, \int_{-\infty}^{+\infty}e^{-\mu \xi^2 t} (4\mu^2 \xi^2 t^2+2\mu t)d\xi = \frac{t}{4\pi^2 x^2} \, \int_{-\infty}^{+\infty}e^{-\mu (t^{\frac{1}{2}}\xi)^2 } \big(4\mu^2 (t^{\frac{1}{2}}\xi)^2 +2\mu\big)d\xi \leq \,  C\, \frac{t^{\frac{1}{2}}}{x^2}.
    \end{split}
\end{equation*}

Once we have established these identities and estimates, returning to (\ref{Iden-Ker-1}), we obtain for $|x|>1$ 
\begin{equation}\label{Iden-Ker-2}
|K(t,x)| \leq C\, \frac{|\sin (\beta t)|}{|x|}+ C\, \frac{t^{\frac{1}{2}}}{|x|^2}\leq C\, (1+t^{\frac{1}{2}})\, \frac{1}{|x|}.
\end{equation}
Moreover, we deduce the following asymptotic profile, which will be used later:
\begin{equation}\label{Profile-Kernel}
K(t,x)=-\frac{\sin(\beta t)}{x}+I(t,x), t>0, \quad \qquad |x|>1.
\end{equation}

On the other hand, for any $x\in \R$ we  can write
\begin{equation}\label{Iden-Ker-3}
|K(t,x)|\leq  \| K(t,\cdot)\|_{L^\infty}\leq \| \widehat{K}(t,\cdot)\|_{L^1}\leq \int_{-\infty}^{+\infty}e^{-\mu \xi^2 t}\, d\xi \leq \frac{C}{t^{\frac{1}{2}}}.
\end{equation}

Therefore, by combining estimates (\ref{Iden-Ker-2}) and (\ref{Iden-Ker-3}), we arrive at the first desired estimate (\ref{Estim-Ker-1}).

\medskip

The  estimate concerning $\mathcal{H}K(t,x)$  follows from similar computations. For the reader’s convenience, we present below the key calculation.
\begin{equation*}
    \begin{split}
\mathcal{H} K(t,x)=&\,  -i \int_{-\infty}^{+\infty} e^{2\pi i x\xi}\,\sign(\xi)\, e^{-\mu \xi^2 t + i\,\sign(\xi) \beta t}d \xi \\
 =&\, \frac{i}{2\pi i x}\int_{-\infty}^{0} \partial_\xi (e^{2\pi i x\xi}) e^{-\mu \xi^2 t - i \beta t}d \xi-\frac{i}{2\pi i x} \int_{0}^{+\infty} \partial_\xi(e^{2\pi i x\xi}) e^{-\mu \xi^2 t + i \beta t}d \xi\\
 =&\,\frac{e^{-i \beta t}+e^{i\beta t}}{2\pi i  x}- \underbrace{\frac{1}{2\pi i  x}\left(\int_{-\infty}^{0} e^{2\pi i x\xi} \partial_\xi (e^{-\mu \xi^2 t - i \beta t})d \xi-\int_{0}^{+\infty} e^{2\pi i x\xi} \partial_\xi (e^{-\mu \xi^2 t + i \beta t})d \xi \right)}_{J(t,x)},
    \end{split}
\end{equation*}
As above, for $x \neq 0$ we have 
\[ \frac{e^{-i \beta t}+e^{i\beta t}}{2\pi i  x}= \frac{\cos(\beta t)}{\pi i x}, \quad \text{and} \quad J(t,x) \leq C\frac{t^{\frac{1}{2}}}{x^2}, \]
which yields the asymptotic profile 
\begin{equation}\label{Profile-Kernel-Hilbert}
\mathcal{H}K(t,x)= \frac{\cos(\beta t)}{\pi i x} + J(t,x), \qquad t>0, \quad  |x|>1. 
\end{equation}
This completes the proof of Proposition \ref{Prop-Kernel-Estim}. 
\end{proof}

\begin{Remarque}\label{Rmk-1} Note that the asymptotic profile (\ref{Profile-Kernel}) shows that, for $\beta\neq 0$ and any time $t \neq \frac{k \pi}{\beta}$ with $k \in \mathbb{Z}$, the kernel 
$K(t,x)$  has the optimal decay rate  $1/|x|$. Therefore, $K(t,\cdot) \notin L^1(\R)$.  Similarly, from the asymptotic profile (\ref{Profile-Kernel-Hilbert}), the same conclusion holds for $\mathcal{H}K(t,x)$  for any time $t \neq \frac{(2k+1) \pi}{2\beta}$. 

\medskip

Consequently, when $\beta \neq 0$,  there is no positive time $t$ for which both  $K(t,x)$ and $\mathcal{H}K(t,x)$ decay at infinity faster than $1/|x|$, and for which $K(t,\cdot)$ and $\mathcal{H}K(t,\cdot)$ belong to $L^1(\R)$. 
\end{Remarque}

\section{Asymptotic behavior in the spatial variable}\label{Sec:Spatial-Decay}  
\subsection{Proof of Theorem \ref{Th-Main}} We begin by outlining the general idea of the proof. To our knowledge, the Hilbert  transform $\mathcal{H}$ is not bounded on the weighted space $L^\infty\big((1+|\cdot|)^{\min(1,\gamma)}dx\big)$,  which makes it difficult to control the nonlinear term $(\mathcal{H}u)\, \partial_x (\mathcal{H}u)$. 

\medskip

To overcome this issue, in the first step of the proof  we introduce an auxiliary coupled system in  the variables $\tilde{u}$ and $\tilde{v}$,  which will allow to control both $u$ and $\mathcal{H}u$ later.  We then prove the existence of local-in-time   $\mathcal{C}_t H^s_x$-solutions, which also satisfy  
\[ \left\| (1+|x|)^{\min(1,\gamma)} (\tilde{u}, \tilde{v})(t,\cdot)\right\|_{L^\infty}<+\infty.  \]

In the second step, we prove the uniqueness of solutions $(\tilde{u}, \tilde{v})$ to this auxiliary system in the larger space $\mathcal{C}_t H^s_x$. 

\medskip

Finally, in the third step, returning to the original equation (\ref{Main-Equation})  we obtain a solution $u \in \mathcal{C}_t H^s_x$, where the pair $(u, \mathcal{H}u)$ also satisfies the auxiliary system introduced above. By uniqueness, it follows that $(u, \mathcal{H}u)=(\tilde{u}, \tilde{v})$, which leads to the pointwise estimate  (\ref{Decay-Solution}). 

\medskip

{\bf First step: the auxiliary system}. In equation (\ref{Main-Equation}), observe that if $\tu(t,x)$ is a solution with initial datum $\tu_0$, then   $\tv(t,x):= \mathcal{H}\tu(t,x)$ formally satisfies:
\[ \partial_t \tv + \mathcal{H}\big( \tu \partial_x \tu\big) + \mathcal{H}\big(\tv \partial_x \tv\big) + \beta \,\mathcal{H}\tv - \mu\, \partial^2_x \tv = 0, \qquad \tv(0,\cdot)=\mathcal{H}\tu_0. \]
Therefore, we will study  both $\tu(t,x)$ and $\tv(t,x)$ as solutions of the following coupled system:
\begin{equation}\label{System}
\begin{cases}\vspace{2mm}
\partial_t \tu + \tu\, \partial_x \tu +  \tv\, \partial_x \tv + \beta \,\mathcal{H}\tu - \mu\, \partial^2_x \tu = 0,\\  \vspace{2mm}
\partial_t \tv + \mathcal{H}\big( \tu \, \partial_x \tu\big) + \mathcal{H}\big(\tv\, \partial_x \tv\big) + \beta \,\mathcal{H}\tv - \mu\, \partial^2_x \tv = 0, \\
\tu(0,\cdot)=\tu_0, \quad \tv(0,\cdot)= \mathcal{H}\tu_0.
\end{cases}
\end{equation}

Let   $T>0$  a  time  to be fixed later. For $s>\frac{3}{2}$, consider the Banach space 
\begin{equation*}
    E_T = \big\{ w \in \mathcal{C}([0,T], H^s(\R)): \, \| w \|_{E_T}<+\infty \big\},
\end{equation*}
endowed with the norm
\begin{equation*}
    \| w \|_{E_T}:= \sup_{0\leq t \leq T}\| w(t,\cdot)\|_{H^s}+ \sup_{0<t\leq T} t^{\frac{1}{2}}\left\| (1+|x|)^{\min(1,\gamma)}\,w(t,\cdot)\right\|_{L^\infty}. 
\end{equation*}
Here, the second term characterizes the pointwise decay in the spatial variable of the function $w(t,x)$. The time weight $t^{\frac{1}{2}}$ is included for technical reasons to ensure the validity of our subsequent estimates.

\medskip

\begin{Proposition}\label{Prop-Tech-1} Assume that $\tu_0$ satisfies (\ref{Def-Data-0}) and (\ref{Def-Data}). There exists a time $\tilde{T}_{0}=\tilde{T}_{0}(\tu_0)>0$ and a solution $(\tu,\tv)\in E_{\tilde{T}_{0}}$ to the coupled system (\ref{System}). In particular, for any $0<t\leq \tilde{T}_{0}$ and $x\in \R$, this solution verifies the pointwise estimate:
	\begin{equation}\label{Decay-Sol-System}
	|\tu(t,x)|+|\tv(t,x)| \leq \frac{\| \tu \|_{E_{\tilde{T}_{0}}}+\| \tv \|_{E_{\tilde{T}_{0}}}}{t^{\frac{1}{2}}(1+|x|)^{\min(1,\gamma)}}.
	\end{equation}
\end{Proposition}	
\begin{proof} We observe that system \eqref{System} can be equivalently reformulated as the following fixed-point problem:
	\begin{equation}\label{System-Mild}
	\begin{cases}\vspace{2mm}
	\tu(t,\cdot)= \ds{K(t,\cdot)\ast \tu_0 - \int_{0}^{t} K(t-\tau,\cdot)\ast \big( \tu\, \partial_x \tu + \tv\, \partial_x \tv \big)(\tau,\cdot)d \tau}, \\
	\tv(t,\cdot)=\ds{ \mathcal{H}K(t,\cdot)\ast \tu_0 - \int_{0}^{t} \mathcal{H}K(t-\tau,\cdot)\ast \big( \tu\, \partial_x \tu + \tv\, \partial_x \tv \big) (\tau,\cdot)d \tau}.
	\end{cases}
	\end{equation}
	In the second equation, the Hilbert transform is applied to the kernel $K(t,\cdot)$, defined in expression (\ref{Kernel}).  We will therefore solve this system in the Banach space $E_T$ introduced above.
	
	\medskip

In the following technical lemmas, we study the linear and nonlinear terms separately. For the linear term involving the initial data, note that by \eqref{Def-Data}, since $\gamma > 0$, we have $\tu_0 \in L^1(\R)$.

\begin{Lemme} Under assumptions \eqref{Def-Data-0} and \eqref{Def-Data}, the following estimate holds:
	\begin{equation}\label{Estim-Linear-LWP}
	\|K(t,\cdot)\ast \tu_0\|_{E_T} + \|\mathcal{H}K(t,\cdot)\ast \tu_0\|_{E_T} \,\leq\, C\,\eta(T)\big(\|\tu_0\|_{H^s} + \|\tu_0\|_{L^1} + C_0\big),
	\end{equation}
	where $C>0$ is a constant depending on the parameters $\beta,\gamma,\mu$, and $\eta(T)>0$ is given in \eqref{Estim-Ker-1}.
\end{Lemme}	
\begin{proof}
From (\ref{Kernel}), for any $t>0$ and $\xi \in \R$ we have $|\widehat{K}(t,\xi)|\leq 1$. Consequently, the first term in the norm $\| \cdot \|_{E_T}$ is directly estimated as:
\begin{equation}\label{Estim-Data-Hs}
\sup_{0\leq t \leq T}\big(\| K(t,\cdot)\ast \tu_0 \|_{H^s}+\|\mathcal{H} K(t,\cdot)\ast \tu_0 \|_{H^s}\big)\leq C\,\| \tu_0 \|_{H^s}.
\end{equation}	

We focus on the second term in $\| \cdot \|_{E_T}$. For the expression  $K(t,\cdot)\ast \tu_0$, for any fixed $x\neq 0$, we write
\begin{equation*}
|K(t,\cdot)\ast \tu_0(x)|\leq \int_{\R} |K(t,x-y)||\tu_0(y)|dy = \underbrace{ \int_{|y|\leq \frac{|x|}{2}}|K(t,x-y)||\tu_0(y)|dy}_{I_1(t,x)} + \underbrace{ \int_{|y|> \frac{|x|}{2}}|K(t,x-y)||\tu_0(y)|dy}_{I_2(t,x)}.  
\end{equation*}	 

For the term $I_1(t,x)$, since $|y|\leq \frac{|x|}{2}$, it follows that $|x-y|\geq |x|-|y| \geq \frac{|x|}{2}$. Then, using the estimate (\ref{Estim-Ker-1}), and since $0<t\leq T $, we obtain 
\begin{equation}\label{Estim-Tech-Kernel}
|K(t,x-y)| \leq C \frac{\eta(t)}{t^{\frac{1}{2}}}\frac{1}{|x-y|}\leq C\, \frac{\eta(T)}{t^{\frac{1}{2}}} \frac{1}{1+|x|}.
\end{equation}
Additionally,  recalling that $\tu_0 \in L^1(\R)$,  we have
\begin{equation}\label{Estim-Decay-Linear-1}
I_1(t,x) \leq \,  C \frac{\eta(T)}{t^{\frac{1}{2}}} \frac{1}{1+|x|}\int_{|y|\leq \frac{|x|}{2}}| \tu_0(y)| dy \leq C\frac{\eta(T)}{t^{\frac{1}{2}}} \frac{1}{1+|x|} \| \tu_0 \|_{L^1}\leq \, C\frac{\eta(T)}{t^{\frac{1}{2}}} \frac{1}{(1+|x|)^{\min(1,\gamma)}} \| \tu_0 \|_{L^1}.
\end{equation}

For the term $I_2(t,x)$,  since $|y|>\frac{|x|}{2}$,  by assumption (\ref{Def-Data})  we have 
\[ |\tu_0 (y)|\leq \frac{C_0}{(1+|y|)^{1+\gamma}}= \frac{C_0}{(1+|y|)(1+|y|)^\gamma}\leq \frac{C_0}{(1+|x|)^\gamma
}\frac{C}{1+|y|}. \]
Fixed a  parameter $1<q<+\infty$. Then  $\left\| \frac{1}{1+|y|}\right\|_{L^q}\leq C_q<+\infty$. Choose  $1<p<+\infty$ such that $1=1/p+1/q$. By (\ref{Estim-Ker-1}), for $0<t\leq T$, we obtain 
\begin{equation}\label{Estim-Kernel-Lp}
\| K(t,\cdot) \|_{L^p} \leq C \frac{\eta(T)}{t^{\frac{1}{2}}}\, \left\| \frac{1}{1+|x|}\right\|_{L^p} \leq  C_p \frac{\eta(T)}{t^{\frac{1}{2}}}.
\end{equation}
Thus, applying the H\"older inequality,  we deduce 
\begin{equation}\label{Estim-Decay-Linear-2}
\begin{split}
&\, I_2(t,x)\leq \, \frac{C_0}{(1+|x|)^\gamma} \int_{|y|>\frac{|x|}{2}} |K(t,x-y)| \frac{1}{1+|y|} dy \leq \, \frac{C_0}{(1+|x|)^\gamma}\, \| K(t,\cdot)\|_{L^p}\, \left\| \frac{1}{1+|y|}\right\|_{L^q} \\
 \leq &\,  \frac{C_0}{(1+|x|)^\gamma}\frac{C_{p,q}\, \eta(T)}{t^{\frac{1}{2}}} \leq  \frac{C_0}{(1+|x|)^{\min(1,\gamma)}}\frac{C\, \eta(T)}{t^{\frac{1}{2}}}.
\end{split}
\end{equation}

Gathering estimates (\ref{Estim-Decay-Linear-1}) and (\ref{Estim-Decay-Linear-2}), we obtain
\begin{equation}\label{Estim-Decay-Linear}
\sup_{0\leq t \leq T} t^{\frac{1}{2}}\left\| (1+|x|)^{\min(1,\gamma)} K(t,\cdot)\ast \tu_0 \right\|_{L^\infty}\leq C\, \eta(T) \big( \| \tu_0 \|_{L^1}+C_0 \big).
\end{equation}

Finally, the expression $\mathcal{H}K(t,\cdot)\ast \tu_0$, follows from the same arguments,  using the pointwise estimate for $\mathcal{H}K(t,x)$ given in (\ref{Estim-Ker-1}).
\end{proof}	

\begin{Lemme} The following estimates hold:
\begin{equation}\label{Estim-Nonlinear-LWP}
\begin{split}
&\, \left\|\int_{0}^{t} K(t-\tau,\cdot)\ast \big( \tu\, \partial_x \tu + \tv\, \partial_x \tv \big) (\tau,\cdot)d \tau \right\|_{E_T}+ \left\|\int_{0}^{t} \mathcal{H} K(t-\tau,\cdot)\ast \big( \tu\, \partial_x \tu + \tv\, \partial_x \tv \big)(\tau,\cdot)d \tau \right\|_{E_T}\\
\leq &\, C\, T^{\frac{1}{1}}\left( 1+  \eta(T)\left(1+T^{\frac{1}{2}}\right) \right) \, \big(\| \tu \|_{E_T}+ \| \tv \|_{E_T}\big)^2.
\end{split}
\end{equation}
	\end{Lemme}	
\begin{proof}
As before, it is enough to consider the expression	$\int_{0}^{t} K(t-\tau,\cdot)\ast \big( \tu\, \partial_x \tu + \tv\, \partial_x \tv \big) (\tau,\cdot)d \tau $. Using known estimates for   $\mathcal{H} K(t-\tau,\cdot)$, the expression $\int_{0}^{t} \mathcal{H} K(t-\tau,\cdot)\ast \big( \tu\, \partial_x \tu + \tv\, \partial_x \tv \big)(\tau,\cdot)d \tau$ follows the same arguments.

\medskip
	
 We start by studying the first term in the norm  $\| \cdot \|_{E_T}$, to this end, we will write $\tu\,\partial_x \tu$ as $\frac{1}{2}\partial_x (\tu^2)$.  First, observe that by (\ref{Kernel}), for any $t>0$ we  have $\ds{\left\|\, |\xi| \widehat{K}(t,\cdot)\right\|_{L^\infty}\leq \frac{C}{t^{\frac{1}{2}}}}$. Additionally, using  well-known properties of the Hilbert transform $\mathcal{H}$, along with the fact that $H^s(\R)$ is an  Banach algebra for $s>\frac{3}{2}$, it follows that:
	\begin{equation}\label{Estim-Nonlinear-Hs}
	\begin{split}
	 &\,\sup_{0\leq t \leq T}\left\| \frac{1}{2} \int_{0}^{t}K(t-\tau,\cdot)\ast \partial_x \big( \tu^2+ \tv^2\big)(\tau,\cdot)\, d \tau \right\|_{H^s} \\
	\leq & \, \sup_{0\leq t \leq T} \int_{0}^{t} \frac{C}{(t-\tau)^{\frac{1}{2}}}\, \big( \| \tu(\tau,\cdot)\|^2_{H^s}+\|\tv(\tau,\cdot)\|^2_{H^s}\big)\, d \tau\\
	\leq &\, C\, T^{\frac{1}{2}}\,  \sup_{0\leq t \leq T} \left(\| \tu(t,\cdot)\|^2_{H^s}+\| \tv(t,\cdot)\|^2_{H^s}\right)\\
	  \leq &\, C\, T^{\frac{1}{2}}(\| \tu\|_{E_T}+\| \tv \|_{E_T})^2. 
	\end{split}
	\end{equation}

Now,  we will focus on the second term in the norm $\| \cdot \|_{E_T}$.  For any fixed $x\neq 0$, we write
	\begin{equation}\label{Identity-Tech-Nonlin}
	\begin{split}
	\int_{0}^{t} K(t-\tau,\cdot)\ast \big( \tu\, \partial_x \tu + \tv\, \partial_x \tv \big) (\tau,x)d \tau =& \underbrace{ \int_{0}^{t} \int_{|y|\leq \frac{|x|}{2}}K(t-\tau, x-y) \big( \tu\, \partial_x \tu + \tv\, \partial_x \tv \big) (\tau,y) dy dt}_{I_1(t,x)} \\
	&\, +\underbrace{ \int_{0}^{t} \int_{|y|> \frac{|x|}{2}} K(t-\tau, x-y) \big( \tu\, \partial_x \tu + \tv\, \partial_x \tv \big) (\tau,y) dy dt}_{I_2(t,x)}.
	\end{split}
	\end{equation}
	where we will estimate each term separately.  
	
	\medskip
	
	For the  term $I_1(t,x)$, since $|y|\leq \frac{|x|}{2}$ and $0<t-\tau <T$, by  estimate (\ref{Estim-Tech-Kernel}) we have
	\[ | K(t-\tau,x-y)|\leq C \frac{\eta(T)}{(t-\tau)^{\frac{1}{2}}}\, \frac{1}{1+|x|}. \]
	Then, using the Cauchy-Schwarz inequality, the continuous embeddings $H^s(\R)\subset H^1(\R)\subset L^2(\R)$ and the fact that $0<t\leq T$,   we find
	\begin{equation}\label{Estim-Tech-Nonlin-1}
	\begin{split}
	I_1(t,x) \leq &\, C\frac{\eta(T)}{1+|x|}\, \int_{0}^{t}\frac{1}{(t-\tau)^{\frac{1}{2}}}\, \int_{|y|\leq \frac{|x|}{2}} \big( | \tu(\tau,y)||\partial_y \tu(\tau,y)|+| \tv(\tau,y)||\partial_y \tv(\tau,y)| \big)dy d\tau \\
	\leq &\,C \frac{\eta(T)}{1+|x|}\, \int_{0}^{t}\frac{1}{(t-\tau)^{\frac{1}{2}}}\, \int_{\R} \big( | \tu(\tau,y)||\partial_y \tu(\tau,y)|+| \tv(\tau,y)||\partial_y \tv(\tau,y)| \big)dy d\tau\\
	\leq &\, C\frac{\eta(T)}{1+|x|}\, \int_{0}^{t}\frac{1}{(t-\tau)^{\frac{1}{2}}}\, \big(\| \tu(\tau,\cdot)\|_{L^2}\, \| \tu (\tau,\cdot)\|_{H^s}+\| \tv(\tau,\cdot)\|_{L^2}\, \| \tv (\tau,\cdot)\|_{H^s} \big)\, d\tau \\
	\leq &\, C \frac{\eta(T)}{1+|x|}\, \left(  \int_{0}^{t}\frac{d \tau }{(t-\tau)^{\frac{1}{2}}} \right)  \sup_{0\leq \tau \leq t}  \big(\| \tu(\tau,\cdot)\|^2_{H^s}+ \| \tv (\tau,\cdot)\|^2_{H^s} \big)\\
	\leq &\, C\frac{\eta(T)}{1+|x|}\, T^{\frac{1}{2}}\, \big( \|\tu\|^2_{E_T}+ \|\tv\|^2_{E_T}\big)\\
	\leq &\, C\frac{\eta(T)T^{\frac{1}{2}}}{(1+|x|)^{\min(1,\gamma)}} \big( \|\tu\|_{E_T}+ \|\tv\|_{E_T}\big)^2.
	\end{split}
	\end{equation}
	
	For the  term $I_2(t,x)$,   since $\tu,\tv \in E_T$ and $|y|>\frac{|x|}{2}$,  we have 
	\[ |\tu(\tau,y)|  \leq  \frac{1}{\tau^{\frac{1}{2}}}\,\frac{1}{(1+|y|)^{\min(1,\gamma)}}\, \| \tu \|_{E_T}   \leq  \frac{C}{\tau^{\frac{1}{2}}}\,\frac{1}{(1+|x|)^{\min(1,\gamma)}}\, \| \tu \|_{E_T},   \]
	and similarly
	\[ |\tv(\tau,y)|  \leq  \frac{C}{\tau^{\frac{1}{2}}}\,\frac{1}{(1+|x|)^{\min(1,\gamma)}}\, \| \tv \|_{E_T}. \]
	Additionally,  note that by the first estimate in (\ref{Estim-Ker-1}), and since $0<t-\tau<T$, it follows that 
	\[ \| K(t-\tau,\cdot)\|_{L^2}\leq C\frac{\eta(T)}{(t-\tau)^{\frac{1}{2}}}.  \] 
	Therefore, using again the Cauchy-Schwarz inequality and the fact that $\ds{\int_{0}^{t}\frac{d \tau}{\tau^{\frac{1}{2}}(t-\tau)^{\frac{1}{2}}}\leq C<+\infty}$,  we can write
	\begin{equation}\label{Estim-Tech-Nonlin-2}
	\begin{split}
	I_2(t,x) \leq &\, C \frac{\eta(T)}{(1+|x|)^{\min(1,\gamma)}}\, \| \tu \|_{E_T} \int_{0}^{t} \frac{1}{\tau^{\frac{1}{2}}}\, \int_{\R} | K(x-y,t-\tau)||\partial_y \tu(\tau,y)| dy d\tau \\
	&\, + C\frac{\eta(T)}{(1+|x|)^{\min(1,\gamma)}}\, \| \tv \|_{E_T} \int_{0}^{t} \frac{1}{\tau^{\frac{1}{2}}}\, \int_{\R} | K(x-y,t-\tau)||\partial_y \tv(\tau,y)| dy d\tau \\
	\leq &\,C \frac{\eta(T)}{(1+|x|)^{\min(1,\gamma)}}\, \| \tu \|_{E_T} \int_{0}^{t}\frac{1}{\tau^{\frac{1}{2}}(t-\tau)^{\frac{1}{2}}}\| \tu (\tau,\cdot)\|_{H^1}d \tau \\
	&\, + C\frac{\eta(T)}{(1+|x|)^{\min(1,\gamma)}}\, \| \tv \|_{E_T} \int_{0}^{t}\frac{1}{\tau^{\frac{1}{2}}(t-\tau)^{\frac{1}{2}}}\| \tv (\tau,\cdot)\|_{H^1}d \tau \\
	\leq &\,C \frac{\eta(T)}{(1+|x|)^{\min(1,\gamma)}}  \big( \|\tu\|_{E_T}+ \|\tv\|_{E_T}\big)^2. 
	\end{split}
	\end{equation}
	
Gathering estimates (\ref{Estim-Tech-Nonlin-1}) and (\ref{Estim-Tech-Nonlin-2}) into identity (\ref{Identity-Tech-Nonlin}), it follows that 
\begin{equation*}
\begin{split}
 &\sup_{0\leq t \leq T}t^{\frac{1}{2}} \left\|(1+|x|)^{\min(1,\gamma)}\int_{0}^{t} K(t-\tau,\cdot)\ast \big( \tu\, \partial_x \tu + \tv\, \partial_x \tv \big) (\tau,\cdot)d \tau \right\|_{L^\infty}\\
 \leq  &\, C T^{\frac{1}{2}}\, \eta(T)(1+T^{\frac{1}{2}}) \big(\| \tu \|_{E_T}+ \| \tv \|_{E_T}\big)^2. 
\end{split}
\end{equation*}	
 \end{proof}

Once estimates \eqref{Estim-Linear-LWP} and \eqref{Estim-Nonlinear-LWP} have been established, we choose a time $\tilde{T}_0 \leq 1$. From the expression of $\eta(\tilde{T}_0)$ in \eqref{Estim-Ker-1}, it follows that $\eta(\tilde{T}_0) \leq C$ and $1+  \eta(\tilde{T}_{0})\left(1+\tilde{T}^{\frac{1}{2}}_{0}\right)\leq C$, for some numerical constant $C>0$. Consequently, estimates \eqref{Estim-Linear-LWP} and \eqref{Estim-Nonlinear-LWP} take the form
\[ 	\|K(t,\cdot)\ast \tu_0\|_{E_{\tilde{T}_0}} + \|\mathcal{H}K(t,\cdot)\ast \tu_0\|_{E_{\tilde{T}_0}} \,\leq\, C\,\big(\|\tu_0\|_{H^s} + \|\tu_0\|_{L^1} + C_0\big),\]
\begin{equation*}
\begin{split}
&\, \left\|\int_{0}^{t} K(t-\tau,\cdot)\ast \big( \tu\, \partial_x \tu + \tv\, \partial_x \tv \big) (\tau,\cdot)d \tau \right\|_{E_{\tilde{T}_0}}+ \left\|\int_{0}^{t} \mathcal{H} K(t-\tau,\cdot)\ast \big( \tu\, \partial_x \tu + \tv\, \partial_x \tv \big)(\tau,\cdot)d \tau \right\|_{E_{\tilde{T}_0}}\\
\leq &\, C\, \tilde{T}^{\frac{1}{2}}_{0} \, \big(\| \tu \|_{E_T}+ \| \tv \|_{E_T}\big)^2.
\end{split}
\end{equation*}
Next,  we choose  $\tilde{T}_0$ sufficiently small, namely
\[ \tilde{T}_0 < \frac{1}{(4 C \big(\| \tu_0 \|_{H^s}+\| \tu_0 \|_{L^1}+ C_0 \big) )^2}, \]
so that the Picard iteration scheme yields a solution $(\tu,\tv) \in E_{\tilde{T}_0}$ to system \eqref{System-Mild}. This completes the proof of Proposition \ref{Prop-Tech-1}. \end{proof} 

 \medskip
 
{\bf Second step: uniqueness of solutions of the system (\ref{System})}.
\begin{Proposition}\label{Prop-Uniqueness} For any fixed time $0<T<+\infty$, the system (\ref{System})  admits a unique solution in the space $\mathcal{C}([0,T], H^s(\R))$, with $s>\frac{3}{2}$. 
\end{Proposition}
\begin{proof}  Suppose $(\tu_1,\tv_1), (\tu_2,\tv_2) \in \mathcal{C}([0,T], H^s(\R))$ are two solutions corresponding to the same initial datum $\tu_0$, and define ${\bf u}:=\tu_1-\tu_2$, ${\bf v}:=\tv_1-\tv_2$.  Then ${\bf u}$ and ${\bf v}$ satisfy:
	\begin{equation*}
	\begin{cases}\vspace{2mm}
	\partial_t {\bf u} + {\bf u}\partial_x \tu_1 + \tu_2 \partial_x {\bf u} + {\bf v}\partial_x \tv_1 + \tv_2 \partial_x {\bf v} + \beta \mathcal{H} {\bf u}-\mu \partial^2_x {\bf u}  = 0, \\ 
	\partial_t {\bf v} + \mathcal{H}({\bf u}\partial_x \tu_1 + \tu_2 \partial_x {\bf u} + {\bf v}\partial_x \tv_1 + \tv_2 \partial_x {\bf v}) + \beta \mathcal{H} {\bf v}-\mu \partial^2_x {\bf v}  = 0.
	\end{cases} 
	\end{equation*}
	Hence,  we obtain:  
	\begin{equation*}
	\begin{split}
	\frac{1}{2}\frac{d}{dt}\| {\bf u}(t,\cdot)\|^2_{L^2} = &\, -\underbrace{\int_{\R} \big( {\bf u}\partial_x u_1\big){\bf u}dx}_{(A_1)}  -\underbrace{\int_{\R}\big( u_2 \partial_x {\bf u} \big) {\bf u}dx}_{(B_1)} - \underbrace{ \int_{\R} \big( {\bf v}\partial_x v_1\big) {\bf u} dx}_{(A_2)}-\underbrace{\int_{\R}\big( v_2 \partial_x {\bf v}\big)  {\bf u} dx}_{(B_2)} \\
	&\, - \mu \| {\bf u}(t,\cdot)\|^2_{\dot{H}^1}, 
	\end{split}
	\end{equation*}
	\begin{equation*}
	\begin{split}
	\frac{1}{2}\frac{d}{dt}\| {\bf v}(t,\cdot)\|^2_{L^2} = &\,  -\underbrace{\int_{\R}\mathcal{H}\big( {\bf u}\partial_x u_1\big){\bf v} dx}_{(A_3)}- \underbrace{\int_{\R}\mathcal{H}\big(u_2 \partial_x {\bf u} \big) {\bf v} dx}_{(B_3)} - \underbrace{\int_{\R}\mathcal{H} \big( {\bf v}\partial_x v_1\big){\bf v}dx}_{(A_4)} - \underbrace{\int_{\R}\mathcal{H}\big( v_2 \partial_x {\bf v}\big) {\bf v} dx}_{(B_4)}\\
	&\,  - \mu \| {\bf v}(t,\cdot)\|^2_{\dot{H}^1}.
		\end{split}
	\end{equation*}
Here, we need to estimate the mixed terms on the right-hand side. Some of them can be handled using similar arguments; therefore, we will only outline the main ideas. 
	
	\medskip
	
	The terms labeled $(A_i)$, for $i=1,\dots,4$, share the generic form $\ds{\int_{\R}({\bf f}\partial_x h){\bf g}dx}$,  with ${\bf f}, {\bf g} \in \{{\bf u}, {\bf v}, \mathcal{H}{\bf v}\}$ and $h \in \{\tu_1, \tv_1\}$. Using H\"older’s inequality, standard properties of the Hilbert transform, and the Sobolev embedding $H^{s-1}(\R) \subset L^\infty(\R)$ (since $s-1>\frac{1}{2}$), it follows that:
	\begin{equation*}
\left| \int_{\R} ({\bf f} \partial_x h){\bf g} dx \right| \leq \| {\bf f}\partial_x h \|_{L^2}\, \| {\bf g}\|_{L^2} \leq \| \partial_x h \|_{L^\infty}\, \|{\bf f}\|_{L^2}\, \| {\bf g }\|_{L^2}\leq C \| h \|_{H^s}(\| {\bf f}\|^2_{L^2}+\| {\bf g}\|^2_{L^2}).   
	\end{equation*}
Thus, we obtain:
\begin{equation}\label{Estim-(A)}
\sum_{i=1}^{4}|(A_i)|\leq C (\| \tu_1 \|_{H^s}+\| \tv_1 \|_{H^s})(\| {\bf u }\|^{2}_{L^2}+\| {\bf v}\|^{2}_{L^2}). 
\end{equation}	

The terms labeled  $(B_i)$ have the generic expression $\ds{\int_{\R}(h \partial_x {\bf f}){\bf g}dx}$, where ${\bf f}$ and  ${\bf g}$ are as above, and  $h\in \{\tu_2,\tv_2\}$. Using arguments analogous to the previous ones, we then write:
\begin{equation*}
\left| \int_{\R}(h \partial_x {\bf f}){\bf g}dx \right| \leq \| h \partial_x {\bf f} \|_{L^2}\, \| {\bf g }\|_{L^2} \leq \|  \partial_x {\bf f} \|_{L^2}\, \| h \|_{L^\infty}\| {\bf g }\|_{L^2} \leq C\, \|  {\bf f} \|_{\dot{H}^1}\, \| h \|_{H^s}\| {\bf g }\|_{L^2}.
\end{equation*}
Applying the discrete Young inequalities, it follows that 
\begin{equation*}
C\, \|   {\bf f} \|_{\dot{H}^1}\, \| h \|_{H^s}\| {\bf g }\|_{L^2} \leq \mu \| {\bf f} \|^2_{\dot{H}^1} + \frac{C}{\mu} \| h \|^2_{H^s}\| {\bf g}\|^{2}_{L^2}.
\end{equation*}
Thus, we obtain
\begin{equation}\label{Estim-(B)}
\sum_{i=1}^{4} |(B_i)| \leq \frac{C}{\mu}(\| \tu_2\|^{2}_{H^s}+\|\tv_2\|^{2}_{H^s})(\| {\bf u}\|^{2}_{L^2}+\| {\bf v }\|^{2}_{L^2})+ \mu \| {\bf u}\|^2_{\dot{H}^1}+ \mu \| {\bf v}\|^2_{\dot{H}^1}. 
\end{equation}

	Combining  estimates (\ref{Estim-(A)}) and (\ref{Estim-(B)}), and applying Gr\"onwall's inequality, we obtain
	\begin{equation*}
	\begin{split}
	\| {\bf {u}}(t,\cdot)\|^2_{L^2}+	\| {\bf {v}}(t,\cdot)\|^2_{L^2}\leq &\big( \| {\bf u}(0,\cdot)\|^2_{L^2}+\|{\bf v}(0,\cdot)\|^{2}_{L^2}\big)\times\\
	&\times\exp\left( C\, \int_{0}^{t}(\| \tu_1(\tau,\cdot)\|_{H^s}+\| \tv_1(\tau,\cdot)\|_{H^s}+\frac{1}{\mu}\| \tu_2(\tau,\cdot)\|^2_{H^s}+\frac{1}{\mu}\| \tv_2(\tau,\cdot)\|^2_{H^s})d \tau\right),
	\end{split}
	\end{equation*}
	where the  integral converges since $\tu_1, \tu_2, \tv_1, \tv_2 \in \mathcal{C}_t H^s_x$. Thus,  as ${\bf u}(0,\cdot)=0$ and ${\bf v}(0,\cdot)=0$, we conclude that $\tu_1=\tu_2$, and $\tv_1= \tv_2$. 
	\end{proof}	

\medskip

{\bf Step 3: End of the proof of Theorem \ref{Th-Main}}. Let $u_0$ be the initial datum satisfying assumptions \eqref{Def-Data-0} and \eqref{Def-Data}.

\medskip

On the one hand, using estimates \eqref{Estim-Data-Hs} and \eqref{Estim-Nonlinear-Hs} (where we write $\mathcal{H}u$ instead of $\tv$), together with the Picard iteration scheme applied to the fixed-point problem
\begin{equation}\label{Equation-Mild}
u(t,\cdot) = K(t,\cdot)\ast u_0 - \int_{0}^{t} K(t-\tau,\cdot)\ast \big( u\, \partial_x u + \mathcal{H}u\, \partial_x (\mathcal{H}u) \big)(\tau,\cdot)\, d\tau,
\end{equation}
we obtain the existence of a time $T'_0 = T'_0(\|u_0\|_{H^s})$, verifying
\begin{equation}\label{Existence-Time-Hs}
T'_0(v_0):= \frac{1}{2}\frac{1}{(4C \| u_0\|_{H^s} )^2},
\end{equation}
 and a solution $u \in \mathcal{C}([0,T'_0], H^s(\R))$ of equation \eqref{Main-Equation}.

\medskip

On the other hand, considering the same initial datum $u_0$, and defining $\tu_0 := u_0$, $\tv_0 := \mathcal{H}u_0$, by Proposition \ref{Prop-Tech-1} there exists $(\tu, \tv) \in E_{\tilde{T}0}$, a solution of system \eqref{System}. Moreover, by the continuous embedding $E_{\tilde{T}_0} \subset \mathcal{C}([0, \tilde{T}_0], H^s(\R))$, we have $(\tu, \tv) \in \mathcal{C}([0, \tilde{T}_0], H^s(\R))$.

\medskip

Finally, observe that $(u, \mathcal{H}u) \in \mathcal{C}([0, T'_0], H^s(\R))$ is also a solution of system \eqref{System}. Defining
\[ T_0 := \min \left(\tilde{T}_0, T'_{0}\right), \]
it follows from Proposition \ref{Prop-Uniqueness} that $\tu = u$ and $\tv = \mathcal{H}u$ on the interval $[0, T_0]$. Since $\tu$ and $\tv$ satisfy the pointwise decay estimates \eqref{Decay-Sol-System}, the desired estimate \eqref{Decay-Solution} holds. This concludes the proof of Theorem \ref{Th-Main}.

\subsection{Proof of Proposition \ref{Prop-Main-1}} Arguing by contradiction, assume that (\ref{Decay-Solution-Fast}) and (\ref{Decay-Solution-Mean-Fast}) hold. Under these assumptions, for any fixed $0 < t \leq T_0$, we will prove that the solution $u(t,x)$ of equation (\ref{Main-Equation}) develops the following asymptotic profile with respect to the spatial variable $x$: 
\[ u(t,x)=\frac{1}{x} \Phi\big(M(u_0),t,u\big)+ o(t)\left(\frac{1}{|x|}\right), \qquad |x|\to +\infty, \]
where the expression $\Phi\big(M(u_0),t,u\big)$, defined in formula (\ref{Phi}) below, does not depend on the variable $x$. This asymptotic profile yields a contradiction with the assumptions (\ref{Decay-Solution-Fast}) and (\ref{Decay-Solution-Mean-Fast}).

\medskip

Returning to the mild formulation  (\ref{Equation-Mild}), we analyze its linear and nonlinear parts separately.

\begin{Lemme}For the linear part, assuming (\ref{Mass-Initial-Datum}),  for any $t>0$ the following asymptotic profile holds:
	\begin{equation}\label{Profile-Linear}
	K(t,\cdot)\ast u_0 (x) = - \frac{\sin(\beta t)}{x}\, M(u_0)  +o(t)\left(\frac{1}{|x|}\right), \qquad |x| > 1,
	\end{equation}
	where the quantity $M(u_0)\neq 0$ is defined in expression (\ref{Mass-Initial-Datum}).
\end{Lemme}	
\begin{proof} For fixed $t>0$ and  $|x|>1$, we write
	\begin{equation}\label{Decomposition} 
	\begin{split}
	K(t,\cdot)\ast u_0 (x)=&\, \int_{\R} K(t,x-y)u_0(y)dy\\
	=&\, K(t,x)\int_{\R}u_0(y)dy+ \int_{\R}\big( K(t,x-y)-K(t,x)\big) u_0(y)dy\\
	=&\, K(t,x)\, M(u_0)+ \int_{\R}\big( K(t,x-y)-K(t,x)\big) u_0(y)dy. 
	\end{split}
	\end{equation}
	
For the first term, using the asymptotic profile (\ref{Profile-Kernel}), together with the fact that, by estimate (\ref{Estim-I}), we have $I(t,x) = o(t)\left(\tfrac{1}{|x|}\right)$, it follows that
\begin{equation*}
K(t,x) M(u_0)
= -\frac{\sin(\beta t)}{x}, M(u_0) + I(t,x) M(u_0)
= -\frac{\sin(\beta t)}{x}\, M(u_0) + o(t)\left(\tfrac{1}{|x|}\right),
\qquad |x|>1.
\end{equation*}

For the second term, we have
\begin{equation}\label{Iden-Tech-1}
\int_{\R}\big( K(t,x-y)-K(t,x)\big) u_0(y)dy = o(t) \left( \frac{1}{|x|}\right), \quad |x|>1. 
\end{equation}
To prove this, we write
\begin{equation}\label{Iden-Tech-2}
\begin{split}
&\, \int_{\R}\big( K(t,x-y)-K(t,x)\big) u_0(y)dy\\
=& \int_{|y|\leq \frac{|x|}{2}}\big( K(t,x-y)-K(t,x)\big) u_0(y)dy + \int_{|y|>\frac{|x|}{2}}\big( K(t,x-y)-K(t,x)\big) u_0(y)dy\\
=& \underbrace{\int_{|y|\leq \frac{|x|}{2}}\big( K(t,x-y)-K(t,x)\big) u_0(y)dy}_{I_1(t,x)} + \underbrace{\int_{|y|>\frac{|x|}{2}} K(t,x-y)u_0(y)dy}_{I_2(t,x)} - \underbrace{K(t,x)\int_{|y|>\frac{|x|}{2}} u_0(y)dy}_{I_3(t,x)},
\end{split}
\end{equation}
so that each term must be estimated separately.

\medskip

To estimate the term $I_1(t,x)$, we use the following pointwise bound:
\begin{equation}\label{Estim-Kernel-Derivative}
|\partial_x K(t,x)| \leq  \frac{C}{x^2}, \qquad t>0, \quad  x\neq 0,
\end{equation}
for some constant $C>0$ depending on $\mu$.  This estimate follows from arguments similar to those used in the proof of (\ref{Estim-Ker-1}). For the reader’s convenience, a detailed proof is provided in Appendix \ref{AppendixA}. 

\medskip

Using (\ref{Kernel}) and well-known properties of the Fourier transform, we deduce that $K(t,\cdot)\in \mathcal{C}^1(\R)$ for any $t>0$. By the first-order  Taylor expansion,   for some $0<\theta<1$ we can write
\begin{equation}\label{Taylor}
K(t,x-y)-K(t,x)= - y \partial_x K(t,x-\theta y). 
\end{equation}
Moreover, since $|y| \leq \frac{|x|}{2}$, it follows that  $|x-\theta y|\geq |x|-\theta |y| \geq | x|-|y| \geq \frac{|x|}{2}$. Therefore, using (\ref{Estim-Kernel-Derivative}) we obtain 
\begin{equation}\label{Estim-Derivative-Kernel}
| \partial_x K(t, x-\theta y)| \leq \frac{C}{|x-\theta y|^2}\leq \frac{C}{x^2}.
\end{equation}
With this  estimate, and the fact that the initial datum satisfies  (\ref{Def-Data}) with $\gamma>1$, the term $I_1(t,x)$ is controlled as 
\begin{equation}\label{I1(t,x)}
\begin{split}
|I_1(t,x)| \leq &\,  \int_{|y|\leq \frac{|x|}{2}} |  K(t,x-y)-K(t,x)| | u_0(y)|dy \leq \int_{|y|\leq \frac{|x|}{2}} |y|| \partial_x K(t, x-\theta y)| \frac{C_0}{(1+|y|)^{1+\gamma}} dy \\
\leq &\,  \frac{C}{x^2}\int_{\R} \frac{dy}{(1+|y|)^\gamma}  \leq  \frac{C}{x^2}=o\left( \frac{1}{|x|}\right).
\end{split} 
\end{equation}

To estimate the term $I_2(t,x)$, note that  since $|y| > \frac{|x|}{2}$ and $u_0(y)$ satisfies (\ref{Def-Data}) (with $\gamma>1$), by  applying the Cauchy-Schwarz inequality and using the estimate (\ref{Estim-Kernel-Lp}) (with $p=2$),  we obtain
\begin{equation}\label{I2(t,x)}
\begin{split}
| I_2(t,x)| \leq &\,  \int_{|y|>\frac{|x|}{2}} | K(t,x-y)| \frac{C_0}{(1+|y|)^{1+\gamma}} dy \leq \frac{C_0}{|x|^\gamma}\int_{\R}|K(t,x-y)| \frac{1}{1+|y|}dy \\
\leq &\,  \frac{C_0}{|x|^\gamma}\| K(t,\cdot)\|_{L^2}\, \left\| \frac{1}{1+|y|} \right\|_{L^2} \leq  \frac{C}{|x|^\gamma}\frac{\eta(t)}{t^\frac{1}{2}}= o(t)\left( \frac{1}{|x|}\right). 
\end{split}
\end{equation}

Finally, to estimate the term $I_3(t,x)$, we use the asymptotic profile for $K(t,x)$ given in (\ref{Profile-Kernel}), recalling that  $I(t,x)=o(t)\left( \frac{1}{|x|}\right)$.  Moreover, since  $u_0(y)$ satisfies (\ref{Def-Data}) (with $\gamma>1$) we have $u_0 \in L^1(\R)$.  By the standard dominated convergence theorem, it follows that  
\[ \lim_{|x|\to +\infty} \int_{|y|>\frac{|x|}{2}}|u_0(y)| dy =0.  \]
Therefore, 
\begin{equation}\label{I3(t,x)}
| I_3(t,x)| \leq \left( \frac{| \sin(\beta t)|}{|x|}+ |I(t,x)| \right)\int_{|y|>\frac{|x|}{2}}|u_0(y)| dy = o(t)\left( \frac{1}{|x|}\right). 
\end{equation}

Having proved estimates (\ref{I1(t,x)}), (\ref{I2(t,x)}), and (\ref{I3(t,x)}), it follows from identity (\ref{Iden-Tech-2}) together with equality (\ref{Iden-Tech-1}) that the desired profile (\ref{Profile-Linear}) holds.
\end{proof}	

\begin{Lemme} For the nonlinear part, assuming (\ref{Decay-Solution-Fast}) and (\ref{Decay-Solution-Mean-Fast}), we obtain that for any $0 < t \leq T_0$ (where the time $T_0 > 0$ is given in Theorem \ref{Th-Main}), the following asymptotic profile holds:
	\begin{equation}\label{Profile-Nonlinear}
	\begin{split}
	&\, \int_{0}^{t} K(t-\tau,\cdot)\ast \big(u\, \partial_x u + \mathcal{H}u\, \partial_x(\mathcal{H}u)\big)(\tau,\cdot)\, d\tau\\
	=& -\frac{1}{x} \left[ \int_{0}^{t} \sin\big(\beta(t-\tau)\big)  \left( \int_{\R}\Big( u\,\partial_y u  +\mathcal{H}u\,\partial_y(\mathcal{H} u)\Big)(\tau, y)dy \right) d\, \tau\right]+ o(t)\left( \frac{1}{|x|}\right),
	\end{split}
	\qquad \,\,\,|x|>2M_\varepsilon. 
	\end{equation}
\end{Lemme}	 
\begin{proof} Note that  it is enough to study the first term $	\int_{0}^{t} K(t-\tau,\cdot)\ast \big(u\, \partial_x u \big)(\tau,\cdot)\, d\tau$, since, by assumptions (\ref{Decay-Solution-Fast}) and (\ref{Decay-Solution-Mean-Fast}), the second term $\int_{0}^{t} K(t-\tau,\cdot)\ast \big(\mathcal{H}u\, \partial_x( \mathcal{H}u) \big)(\tau,\cdot)\, d\tau$ can be treated by the  same arguments.  
	
\medskip

Following the same decomposition used in identities (\ref{Decomposition}) and (\ref{Iden-Tech-1}), for  fixed $0<t\leq T_0$ and  $|x|>2M_\varepsilon$,  we write 
\begin{equation}\label{Iden-Tech-3}
\begin{split}
&\int_{0}^{t}\int_{\R} K(t-\tau, x-y) u(\tau, y)\partial_y u(\tau,y) dy\, d\tau\\
=&\, \underbrace{ \int_{0}^{t} K(t-\tau,x) \left( \int_{\R} u(\tau,y)\partial_y u(\tau,y) dy \right) d\tau}_{I_1(t,x)} \\
&\, + \underbrace{ \int_{0}^{t} \int_{|y|\leq \frac{|x|}{2}} \big( K(t-\tau, x-y)- K(t-\tau,x) \big) u(\tau,y)\partial_y u(\tau,y) dy \, d \tau}_{I_2(t,x)} \\
&\, + \underbrace{ \int_{0}^{t} \int_{|y|>\frac{|x|}{2}}  K(t-\tau, x-y) u(\tau,y)\partial_y u(\tau,y) dy \, d \tau}_{I_3(t,x)} \\
&\, - \underbrace{ \int_{0}^{t}K(t-\tau, x) \left( \int_{|y|>\frac{|x|}{2}} u(\tau,y)\partial_y u(\tau,y) dy\right) d \tau}_{I_4(t,x)},
\end{split}
\end{equation}
where each term on the right-hand side must be analyzed separately. Without loss of generality, we may assume that $2M\varepsilon > 1$.

\medskip

For the first term, using the identity (\ref{Profile-Kernel}),  we obtain
\begin{equation*}
I_1(t,x)=  -\frac{1}{x} \int_{0}^{t} \sin\big(\beta(t-\tau)\big)  \left( \int_{\R} u(\tau, y)\partial_y u(\tau,y) dy\right) d\, \tau+\int_{0}^{t}I(t-\tau,x)\left( \int_{\R} u(\tau, y)\partial_y u(\tau,y) dy\right) d\, \tau,
\end{equation*}
where, for the second expression above,  using estimate (\ref{Estim-I}), the Cauchy-Schwarz inequality, and the continuous embeddings $H^s(\R) \subset L^2(\R)$ and $H^s(\R)\subset \dot{H}^1(\R)$, since $s>\frac{3}{2}$, it follows that 
\begin{equation*}
\begin{split}
&\,  \left| \int_{0}^{t}K(t-\tau, x)\left( \int_{\R} u(\tau, y)\partial_y u(\tau,y) dy\right)d\tau \right|  \leq \, \frac{C}{x^2} \int_{0}^{t} (t-\tau)^{\frac{1}{2}} \| u (\tau, \cdot)\|_{L^2}\, \| u(\tau, \cdot)\|_{\dot{H^1}}\, d\tau \\
\leq &\, \frac{C\, t^{\frac{3}{2}}}{x^2} \left( \sup_{0\leq \tau \leq T_0}\| u(\tau,\cdot)\|^2_{H^s}\right).
\end{split}
\end{equation*}
Therefore, we obtain 
\begin{equation}\label{I1(t,x)-Nonlin}
I_1(t,x)= -\frac{1}{x} \int_{0}^{t} \sin\big(\beta(t-\tau)\big)  \left( \int_{\R} u(\tau, y)\partial_y u(\tau,y) dy\right) d\, \tau+ o(t)\left( \frac{1}{|x|}\right). 
\end{equation}

For the second term, from identity (\ref{Taylor}) and  estimate (\ref{Estim-Derivative-Kernel}),  applying the Cauchy–Schwarz inequality together with assumption (\ref{Decay-Solution-Mean-Fast}), it follows that 
\begin{equation}\label{I2(t,x)-Nonlin} 
\begin{split}
I_2(t,x) \leq  &\, \frac{C}{x^2} \int_{0}^{t} \int_{|y|\leq \frac{|x|}{2}} |y| |u(\tau, y)||\partial_y u(\tau,y)| dy\, d \tau \leq \frac{C}{x^2} \int_{0}^{t} \int_{\R} |y| |u(\tau, y)||\partial_y u(\tau,y)| dy\, d \tau\\
\leq &\, \frac{C\,t}{x^2}  \left( \sup_{0\leq \tau \leq T_0} \|\, |x|u(\tau,\cdot) \|_{L^2} \right) \left( \sup_{0\leq \tau \leq T_0} \| u(\tau,\cdot) \|_{H^s} \right)\\
=&\, o(t)\left( \frac{1}{|x|}\right).
 \end{split}
\end{equation}

For the third term, by assumption (\ref{Decay-Solution-Fast}) and the fact that $|y|\geq \frac{|x|}{2}\geq M_\varepsilon$, together with the Cauchy-Schwarz inequality and estimate (\ref{Estim-Kernel-Lp}), we have
\begin{equation}\label{I3(t,x)-Nonlin}
\begin{split}
I_3(t,x)\leq &\, \int_{0}^{t}\int_{|y|>\frac{|x|}{2}} |K(t-\tau, x-y)|\, \frac{C_2}{\tau^{\frac{1}{2}}|y|^{1+\varepsilon}}| \partial_y u (\tau,y)|dy \\
\leq &\, \frac{C_2}{|x|^{1+\varepsilon}}\int_{0}^{t} \frac{1}{\tau^{\frac{1}{2}}} \left( \int_{\R}|K(t-\tau, x-y)||\partial_y u(\tau,y)| dy\right)\, d\tau \\
\leq &\, \frac{C_2}{|x|^{1+\varepsilon}}\int_{0}^{t} \frac{1}{\tau^{\frac{1}{2}}} \| K(t-\tau)\|_{L^2}\, \| \partial_y u(\tau,\cdot)\|_{L^2}\, d\tau \\
\leq &\, \frac{C}{|x|^{1+\varepsilon}}\left( \int_{0}^{t} \frac{d\tau}{\tau^{\frac{1}{2}}(t-\tau)^{\frac{1}{2}}} \right)\left( \sup_{0\leq \tau \leq T_0} \|u(\tau,\cdot) \|_{H^s} \right)\\
\leq &\, \frac{C}{|x|^{1+\varepsilon}} \left( \sup_{0\leq \tau \leq T_0} \|u(\tau,\cdot) \|_{H^s} \right)\\
=&\, o \left(\frac{1}{|x|}\right).
\end{split}
\end{equation}

Finally, since $u(\tau,\cdot)\,\partial_y u(\tau,\cdot) \in L^1(\R)$ for all $0 < \tau < t$, similar arguments as in (\ref{I3(t,x)}) yield
\begin{equation}\label{I4(t,x)-Nonlin}
I_4(t,x) = o(t)\left(\frac{1}{|x|}\right).
\end{equation}

Gathering the identities and estimates (\ref{I1(t,x)-Nonlin}), (\ref{I2(t,x)-Nonlin}), (\ref{I3(t,x)-Nonlin}), and (\ref{I4(t,x)-Nonlin}) into identity (\ref{Iden-Tech-3}), the desired asymptotic profile (\ref{Profile-Nonlinear}) follows. 
\end{proof}	

{\bf End of the proof of Proposition \ref{Prop-Main-1}}. Returning to the mild formulation (\ref{Equation-Mild}) and using the asymptotic profiles (\ref{Profile-Linear}) and (\ref{Profile-Nonlinear}), for any $0 < t \leq T_0$ and $|x| > 2M_\varepsilon$, the solution $u(t,x)$ satisfies
\[u(t,x)= \frac{1}{x}\left[ -\sin(\beta t)M(u_0)-  \int_{0}^{t} \sin\big(\beta(t-\tau)\big)  \left( \int_{\R}\Big( u\,\partial_y u  +\mathcal{H}u\,\partial_y(\mathcal{H} u)\Big)(\tau, y)dy \right) d\, \tau\right]+o(t)\left(\frac{1}{|x|}\right), \]
where, defining 
\begin{equation}\label{Phi}
\Phi\big(M(u_0),t,u\big)=: -\sin(\beta t)M(u_0)-  \int_{0}^{t} \sin\big(\beta(t-\tau)\big)  \left( \int_{\R}\Big( u\,\partial_y u  +\mathcal{H}u\,\partial_y(\mathcal{H} u)\Big)(\tau, y)dy \right) d\, \tau,
\end{equation}
we can write
\[ u(t,x)=\frac{1}{x} \Phi\big(M(u_0),t,u\big)+ o(t)\left(\frac{1}{|x|}\right), \qquad |x|>2M_\varepsilon.  \]
From this profile of $u(t,x)$, we can derive a contradiction, which invalidates both assumptions (\ref{Decay-Solution-Fast}) and (\ref{Decay-Solution-Mean-Fast}).
\medskip

We write
\begin{equation*}
\left|  u(t,x)\right|= \left| \frac{1}{x} \Phi\big(M(u_0),t,u) - \left( - o(t)\left(\frac{1}{|x|}\right)\right) \right|  \geq \frac{1}{|x|}\left|\Phi(u_0,t,u) \right|- \left| o(t)\left(\frac{1}{|x|}\right)\right|. 
\end{equation*}
By the definition of the term $o(t)\left(\frac{1}{|x|}\right)$, for $\left|\Phi\big(M(u_0),t,u\big) \right|>0$ there exists $M_t>0$ sufficiently large so that $\left| o(t)\left(\frac{1}{|x|}\right)\right| < \frac{1}{2|x|}\left|\Phi\big(M(u_0),t,u\big) \right|$, for any $|x|>M_t$. Hence, for $|x|>\max(2M_\varepsilon,M_t)$, we obtain 
\begin{equation}\label{Estim-Contradiction}
\frac{1}{2|x|}|\Phi\big(M(u_0),t,u\big)| \leq |u(t,x)|.
\end{equation}

Fixing the time $t=\frac{T_0}{2}$,  assumption (\ref{Decay-Solution-Fast}) yields 
\[ \frac{1}{2}\left|\Phi\left(M(u_0),\frac{T_0}{2},u\right)\right|\frac{1}{|x|} \leq \frac{C_2}{\left(\frac{T_0}{2}\right)^{\frac{1}{2}}} \frac{1}{|x|^{1+\varepsilon}}, \qquad |x|> \max\left( 2M_\varepsilon, M_{\frac{T_0}{2}}\right),  \]
which gives us a contradiction by letting $|x|\to +\infty$. 

\medskip

On the other hand, from (\ref{Estim-Contradiction}) it follows that
\[ \frac{1}{2}|\Phi\big(M(u_0),t,u)| \leq |x|\, |u(t,x)|,  \qquad |x|> \max\left( 2M_\varepsilon, M_{t}\right), \]
and hence, for any fixed time $0<t\leq T_0$, we obtain $u(t,\cdot)\notin L^2\big(\R, |x|dx\big)$,  which contradicts assumption (\ref{Decay-Solution-Mean-Fast}).Proposition \ref{Prop-Main-1} has now been proven.

\section{Global in time properties}\label{Sec:Global-in-time}
\subsection{Proof of Proposition \ref{Prop-Main-2}}  From equation (\ref{Main-Equation}), we get:
\begin{equation*}
\frac{1}{2}\frac{d}{dt}\| u(t,\cdot)\|^2_{L^2}=\, - \int_{\R}(\mathcal{H}u) \partial_x (\mathcal{H}u)\, u \, dx - \mu \| u(t,\cdot)\|^2_{\dot{H}^1}
\leq \, - \int_{\R}(\mathcal{H}u) \partial_x (\mathcal{H}u)\, u \, dx.
\end{equation*}
To control the  term on the right-hand-side,  we write:
\begin{equation*}
- \int_{\R}(\mathcal{H}u) \partial_x (\mathcal{H}u)\, u \, dx=\frac{1}{2}\int_{\R}(\mathcal{H}u)^2\, \partial_x u\, dx \leq \|\partial_x u\|_{L^\infty}\, \| u \|^2_{L^2},   
\end{equation*}
and applying  Gr\"onwall's inequality, we obtain:
\begin{equation}\label{Gronwall}
\| u(t,\cdot)\|^2_{L^2}\leq \| u_0 \|^2_{L^2}\, \exp\left( \int_{0}^{t}\| \partial_x u(\tau,\cdot)\|_{L^\infty}\, d \tau \right).
\end{equation}

Using this estimate, we establish the blow-up criterion (\ref{Blow-up-Criterion}) as follows. First, assume that
\begin{equation}\label{Assumption1}
\lim_{t \to T_{*}}\| u(t,\cdot)\|_{H^s}=+\infty.
\end{equation}
Arguing by contradiction, suppose that
\begin{equation}\label{Control-Integral}
\int_{0}^{T_{*}}\| \partial_x u(t,\cdot)\|_{L^\infty}\, d t <+\infty.
\end{equation}
Define 
\begin{equation*}
M_{0}:= \| u_0\|^2_{L^2}\,\exp\left( \int_{0}^{T_{*}}\| \partial_x u(t,\cdot)\|_{L^\infty}\, dt  \right)<+\infty.   
\end{equation*}
Then,   from estimate (\ref{Gronwall}), for all $0<t\leq T_{*}$ we obtain
\begin{equation}\label{Control-L2}
\| u(t,\cdot)\|^2_{L^2}\leq M_0.
\end{equation}

\medskip

Now, recall that for any initial datum $v_0\in H^s(\R)$ with $s>\frac{3}{2}$, the existence time $T_0=T_0(v_0)>0$ of the corresponding  solution $v\in \mathcal{C}([0,T_0], H^s(\R))$ to equation (\ref{Main-Equation}) is given by the expression (\ref{Existence-Time-Hs}):
\begin{equation*}
T_0(v_0):= \frac{1}{2}\frac{1}{(4C \| v_0\|_{H^s} )^2}\leq \frac{1}{2}\frac{1}{(4C \| v_0\|_{L^2} )^2}.
\end{equation*}
Consequently, the time $T_0(v_0)$ is a decreasing function of $\| v_0 \|^2_{L^2}$. This property implies that there exists a time  $0<T_1<T_{*}$ such that  for any initial datum satisfying $\| v_0 \|^2_{L^2}\leq M_0 $, the associated solution $v\in \mathcal{C}_t H^s_x$ exists at least on the interval  $[0,T_1]$, and satisfies $v \in \mathcal{C}([0,T_1], H^s(\R))$. 

\medskip

In this context, fix  $0<\varepsilon<T_1$, and  consider the initial datum $v_0:= u(T_{*}-\varepsilon,\cdot)$, which  satisfies  $\| v_0\|^2_{L^2}$ by the bound in  (\ref{Control-L2}). As mentioned above, the corresponding solution $v(t,x)$ in defined over the interval $[0,T_1]$.

\medskip

Therefore, gathering the solutions $u(t,x)$ and $v(t,x)$, which arise from the initial datum $u_0$ and $v_0$ respectively, we  define the function
\begin{equation*}
\Tilde{u}(t,\cdot)= \begin{cases}\vspace{2mm}
u(t,\cdot),&  \text{for}\quad 0 \leq t \leq T_{*}-\varepsilon, \\
v(t,\cdot),&  \text{for}\quad T_{*}-\varepsilon \leq t \leq T_{*}-\varepsilon+T_1.
\end{cases}
\end{equation*}
Since $0<\varepsilon<T_1$ it follows that $T_{*}-\varepsilon+T_1 >T_{*}$, implying that $\Tilde{u}(t,\cdot)$ is a solution of equation (\ref{Main-Equation}) with initial datum $u_0$,  defined on the extended interval  $[0,T_{*}-\varepsilon+T_1]$,  satisfies $\Tilde{u} \in \mathcal{C}([0,T_{*}-\varepsilon+T_1], H^s(\R))$. This contradicts assumption (\ref{Assumption1}). 

\medskip

Now, suppose that
\begin{equation}\label{Assumption2}
\int_{0}^{T_{*}}\| \partial_x u(t,\cdot)\|_{L^\infty}\, dt =+\infty. 
\end{equation}
Again, arguing by contradiction, assume that
\begin{equation*}
\lim_{t\to T_{*}}\| u(t,\cdot)\|_{H^s}<+\infty.
\end{equation*}
Then, the solution $u(t,x)$ of equation (\ref{Main-Equation}) can be extended beyond time $T_{*}$, and there exists an $\varepsilon>0$ such that $u\in \mathcal{C}([0,T_{*}+\varepsilon], H^s(\R))$. 

\medskip

Since $s>\frac{3}{2}$, applying the Sobolev embedding yields
\begin{equation*}
\int_{0}^{T_{*}}\| \partial_x u(t,\cdot)\|_{L^\infty}\, dt \leq T_{*}\left( \sup_{0\leq t \leq T_{*}+\varepsilon}\| u(t,\cdot)\|_{H^s}\right) <+\infty,
\end{equation*}
which contradicts (\ref{Assumption2}).   Therefore,  Proposition \ref{Prop-Main-2} is proven. 

\subsection{Proof of Proposition \ref{Prop-Main-3}} Assume that (\ref{Condition-L1-Solution}) holds. Using this fact, we will show that the pointwise estimate (\ref{Decay-Solution}) holds for all times $0 < t \leq T_*$. Specifically,  define
\[ g(t):= t^{\frac{1}{2}}  \left( \| (1+|x|)^{\min(1,\gamma)} u(t,\cdot)\|_{L^\infty} +  \| (1+|x|)^{\min(1,\gamma)} \mathcal{H}u(t,\cdot)\|_{L^\infty}  \right), \]
and we will  show that $g(t) < +\infty$ for every  $0 < t \leq T_*$.

\medskip

Using the  mild formulation of the solution $u(t,x)$ given in (\ref{Equation-Mild}), we have 
\[ u(t,\cdot) = K(t,\cdot)\ast u_0 - \int_{0}^{t} K(t-\tau,\cdot)\ast \big( u\, \partial_x u + \mathcal{H}u\, \partial_x (\mathcal{H}u) \big)(\tau,\cdot)\, d\tau,  \]
and 
\[ \mathcal{H}u(t,\cdot)=  \mathcal{H}K(t,\cdot)\ast u_0 - \int_{0}^{t} \mathcal{H}K(t-\tau,\cdot)\ast \big( u\, \partial_x u + \mathcal{H}u\, \partial_x (\mathcal{H}u) \big)(\tau,\cdot)\, d\tau. \]
Since $u(t,\cdot)$ and $\mathcal{H}u(t,\cdot)$, as well as  $K(t,\cdot)$ and $\mathcal{H}K(t,\cdot)$, satisfy the same estimates,  we introduce the unified notation
\begin{equation}\label{Mild-Unified}
\bu(t,\cdot)= \bK(t,\cdot)\ast u_0 - \int_{0}^{t} \bK(t-\tau,\cdot)\ast \bu \, \partial_x \bu(\tau,\cdot)\, d\tau, 
\end{equation}
where $\bu \in \{ u, \mathcal{H}u \}$ and $\bK \in \{ K, \mathcal{H}K \}$, for simplicity of presentation. 

\medskip

With a slight abuse of notation, we rewrite
\begin{equation}\label{g}
g(t):= t^{\frac{1}{2}} \| (1+|x|)^{\min(1,\gamma)} \bu(t,\cdot)\|_{L^\infty}.
\end{equation}
Using the integral formulation above, we obtain 
\begin{equation}\label{Estim-g}
\begin{split}
g(t) \leq  \underbrace{t^{\frac{1}{2}}\| (1+|x|)^{\min(1,\gamma)} \bK(t,\cdot)\ast u_0 \|_{L^\infty}}_{g_1(t)}+ \underbrace{ t^{\frac{1}{2}}  \left\| (1+|x|)^{\min(1,\gamma)} \int_{0}^{t} \bK(t-\tau,\cdot)\ast \bu\, \partial_x \bu(\tau,\cdot)\, d \tau \right\|_{L^\infty}}_{g_2(t)}, 
\end{split}
\end{equation} 
where each term on the right-hand side must be treated separately. 

\medskip

For the first term, by estimate (\ref{Estim-Decay-Linear}), for any $0<t\leq T_*$ we directly have
\begin{equation}\label{g1}
g_1(t) \leq C\, \eta(T_*)\big( \| u_0 \|_{L^1}+C_0 \big)=: \mathfrak{C}_0.
\end{equation}

For the second term, using the decomposition as in (\ref{Identity-Tech-Nonlin}), we  obtain
\begin{equation}\label{Identity-Tech-Nonlin-bis}
\begin{split}
&\, \int_{0}^{t} \bK(t-\tau,\cdot)\ast  \bu\, \partial_x \bu  (\tau,x)d \tau \\
=&\,  \underbrace{ \int_{0}^{t} \int_{|y|\leq \frac{|x|}{2}}\bK(t-\tau, x-y)  \bu\, \partial_x \bu (\tau,y) dy dt}_{I_1(t,x)}  +\underbrace{ \int_{0}^{t} \int_{|y|> \frac{|x|}{2}} \bK(t-\tau, x-y) \bu\, \partial_x \bu  (\tau,y) dy dt}_{I_2(t,x)}.
\end{split}
\end{equation}

The term $I_1(t,x)$ has already been estimated in (\ref{Estim-Tech-Nonlin-1}) (with $\bu$ in place of $\tilde{u}$ and $ \tilde{v}$), yielding for $0<t \leq T_*$,
\begin{equation}\label{I1(t,x)-bis}
\begin{split}
I_1(t,x) \leq &\,  C \frac{\eta(T_*)}{1+|x|} \left( \int_{0}^{t} \frac{d \tau}{(t-\tau)^{\frac{1}{2}}} \right)\sup_{0\leq \tau \leq t} \, \| \bu(\tau, \cdot)\|^2_{H^s}  \\
\leq &\,  C \frac{\eta(T_*) \, T^{\frac{1}{2}}_{*}}{(1+|x|)^{\min(1,\gamma)}} \left(  \sup_{0\leq \tau \leq T_*}  \| \bu(\tau, \cdot)\|^2_{H^s}\right)\\
=:&\, \frac{\mathfrak{C}_1}{(1+|x|)^{\min(1,\gamma)}}.
\end{split}
\end{equation}

For the term $I_2(t,x)$, using  $g(t)$ as  defined in (\ref{g}),  we write
\begin{equation*}
\begin{split}
I_2(t,x) \leq &\, \int_{0}^{t}\int_{|y|>\frac{|x|}{2}} |\bK(t-\tau, x-y)|  \frac{\tau^{\frac{1}{2}}(1+|y|)^{\min(1,\gamma)}\,|\bu(\tau,y)|}{\tau^{\frac{1}{2}}(1+|y|)^{\min(1,\gamma)}}|\partial_y \bu(\tau, y)|  \\
\leq &\, \frac{C}{(1+|x|)^{\min(1,\gamma)}} \int_{0}^{t} \frac{g(\tau)}{\tau^{\frac{1}{2}}} \int_{\R} |\bK(t-\tau, x-y)| \, |\partial_y \bu(\tau, y)|dy\, d\tau.
\end{split}
\end{equation*}
Applying the Cauchy–Schwarz inequality and the estimate (\ref{Estim-Kernel-Lp}) with $p=2$, we obtain
\begin{equation}\label{I2(t,x)-bis}
\begin{split}
I_2(t,x) \leq &\, \frac{C}{(1+|x|)^{\min(1,\gamma)}} \int_{0}^{t} \frac{g(\tau)}{\tau^{\frac{1}{2}}} \| \bK(t-\tau,\cdot)\|_{L^2} \,\| \partial_y \bu(\tau, \cdot)\|_{L^2} d\tau\\
\leq &\,  \frac{C\, \eta(T_*)}{(1+|x|)^{\min(1,\gamma)}}\,  \int_{0}^{t}  \frac{1}{(t-\tau)^{\frac{1}{2}}\tau^{\frac{1}{2}}} g(\tau) \| \bu(\tau, \cdot)\|_{H^s}d \tau\\
\leq &\,  \frac{C\, \eta(T_*)}{(1+|x|)^{\min(1,\gamma)}} \left( \sup_{0\leq \tau \leq T_*}\| \bu(\tau,\cdot)\|_{H^s} \right)  \int_{0}^{t}  \frac{1}{(t-\tau)^{\frac{1}{2}}\tau^{\frac{1}{2}}} g(\tau) d \tau\\
=:& \,  \frac{\mathfrak{C}_2}{(1+|x|)^{\min(1,\gamma)}}\,  \int_{0}^{t}  \frac{1}{(t-\tau)^{\frac{1}{2}}\tau^{\frac{1}{2}}} g(\tau) d \tau.
\end{split}
\end{equation}

Combining estimates (\ref{I1(t,x)-bis}) and (\ref{I2(t,x)-bis}) in (\ref{Identity-Tech-Nonlin-bis}), we obtain, for $0<t\leq T_*$, 
\begin{equation}\label{g2}
\begin{split}
g_2(t) \leq &\, t^{\frac{1}{2}}\, \mathfrak{C}_1 + t^{\frac{1}{2}} \mathfrak{C}_2 \int_{0}^{t}  \frac{1}{(t-\tau)^{\frac{1}{2}}\tau^{\frac{1}{2}}} g(\tau) d \tau \leq \,   T^{\frac{1}{2}}_{*}\, \mathfrak{C}_1  + t^{\frac{1}{2}}\mathfrak{C}_2\, \int_{0}^{t}  \frac{1}{(t-\tau)^{\frac{1}{2}}\tau^{\frac{1}{2}}} g(\tau) d \tau \\
=:&\, \mathfrak{C}_3 +t^{\frac{1}{2}} \, \mathfrak{C}_2\,  \int_{0}^{t}  \frac{1}{(t-\tau)^{\frac{1}{2}}\tau^{\frac{1}{2}}} g(\tau) d \tau.
\end{split}
\end{equation}

Finally, combining (\ref{g1}) and (\ref{g2}) in (\ref{Estim-g}), we conclude that
\begin{equation*}
g(t) \leq  \mathfrak{C}_0 + \mathfrak{C}_3 + t^{\frac{1}{2}} \, \mathfrak{C}_2\,  \int_{0}^{t}  \frac{1}{(t-\tau)^{\frac{1}{2}}\tau^{\frac{1}{2}}} g(\tau) d \tau =: \mathfrak{C}_4+ t^{\frac{1}{2}} \, \mathfrak{C}_2\,  \int_{0}^{t}  \frac{1}{(t-\tau)^{\frac{1}{2}}\tau^{\frac{1}{2}}} g(\tau) d \tau.
\end{equation*}

Once we have this estimate for $g(t)$, we invoke the following Gr\"onwall-type inequality. For a proof, see \cite[Lemma 3.4]{Brandolese}.

\begin{Lemme}\label{Lemma-Gronwall} Let $T>0$ and $g: [0,T]\to \R$ be a non-negative and locally bounded function, which  for any $0<t\leq T$, satisfies:
	\begin{equation}\label{Gronwall-2}
	g(t) \leq M_1 + M_2 \int_{0}^{t} \frac{1}{(t-\tau)^{\frac{1}{2}}\tau^{\frac{1}{2}}}g(\tau) d \tau,
	\end{equation}
	for some constants $M_1, M_2>0$ possibly depending on  $T$. Additionally, define
	\[ M:= \int_{0}^{1}\frac{d\tau}{(1-\tau)^{\frac{1}{2}}\tau^{\frac{1}{2}}}. \]
	
	If $M_2 < \frac{1}{M}$, then it follows that $g(t) \leq M_1$ for every $0<t\leq T$. 
\end{Lemme}	

Within the framework of this lemma, we choose a time  $0<T_1 \leq T_*$, and set $M_1= \mathfrak{C}_4$ and $M_2= T^{\frac{1}{2}}_1\mathfrak{C}_2$. Thus, for any $0<t\leq T_1$, the function $g(t)$ defined in (\ref{g}) satisfies 
\[ g(t) \leq  \mathfrak{C}_4+ T^{\frac{1}{2}}_1 \, \mathfrak{C}_2\,  \int_{0}^{t}  \frac{1}{(t-\tau)^{\frac{1}{2}}\tau^{\frac{1}{2}}} g(\tau) d \tau,\]
 which corresponds to the estimate (\ref{Gronwall-2}). 

\medskip

We then fix  $0<T_1\leq T_*$ sufficiently small so that $T^{\frac{1}{2}}_1\mathfrak{C}_2 < \frac{1}{M}$. Consequently, it follows that $g(t)\leq \mathfrak{C}_4$ for every $0<t \leq T_1$. 

\medskip

Thereafter, we can iterate this process up to the time $T_*$ as follows. We now consider the initial datum $u(T_1,\cdot)$, which, by assumption (\ref{Condition-L1-Solution}), satisfies ${\bf u}(T_1,\cdot) \in L^1(\mathbb{R})$.

\begin{Remarque}\label{Rmk-2}
	Returning to the mild formulation (\ref{Mild-Unified}), recall from Remark \ref{Rmk-1} that ${\bf K}(t,\cdot) \notin L^1(\mathbb{R})$. Consequently, even assuming that $u_0 \in L^1(\mathbb{R})$, it is not clear whether this information allows us to deduce that ${\bf K}(t,\cdot) \ast u_0 \in L^1(\mathbb{R})$. Therefore, to the best of our knowledge, it remains open whether the $L^1$-integrability of the initial data persists for the associated solution.
\end{Remarque}

For any time $t \geq T_1$, we then consider the fixed-point equation
\begin{equation*}
{\bf u}(t,\cdot)= K(t-T_1, \cdot)\ast {\bf u}(T_1,\cdot)- \int_{T_1}^{t}K(t-\tau,\cdot) \ast {\bf u}\, \partial_x {\bf u}(\tau,\cdot)\, d \tau. 
\end{equation*}

Since ${\bf u}(T_1,\cdot) \in H^s \cap L^1(\mathbb{R})$, estimates (\ref{g1}) and (\ref{g2}) remain valid. Using again Lemma \ref{Lemma-Gronwall}, we obtain,  for some constant $\mathfrak{C}^{'}_{4}>0$ and for  a sufficiently small time $T_2 > 0$, that $g(t) \leq \mathfrak{C}'_{4}$ for any $T_1 \leq t \leq T_2$.

\medskip

Finally, by repeating this procedure a finite number of times, we reach the time $T_*$. This concludes the proof of Proposition \ref{Prop-Main-3}.

\appendix
\section{Proof of estimate (\ref{Estim-Kernel-Derivative})}\label{AppendixA} 
Using well-known properties of the Fourier transform, along with identity (\ref{Kernel}) and the fact that $\partial_\xi e^{2\pi i x\xi}= 2\pi i x \, e^{2\pi i x\xi}$, integrating by parts we write
\begin{equation*}
\begin{split}
\partial_x K(t,x)=&\, \int_{-\infty}^{0} e^{2\pi i x \xi} i\xi e^{-\mu \xi^2 t-i\beta t}\, d\xi  + \int_{0}^{+\infty} e^{2\pi i x\xi} i \xi e^{-\mu \xi^2 t+i\beta t}\, d\xi \\
=&\,  \frac{1}{2\pi i x}\int_{-\infty}^{0} \partial_{\xi}(e^{2\pi i x \xi}) i\xi e^{-\mu \xi^2 t-i\beta t}\, d\xi  + \frac{1}{2\pi i x} \int_{0}^{+\infty} \partial_{\xi}(e^{2\pi i x\xi}) i \xi e^{-\mu \xi^2 t+i\beta t}\, d\xi\\
=&- \frac{1}{2\pi x} \int_{-\infty}^{0} e^{2\pi i x \xi} e^{-\mu \xi^2 t-i\beta t}(1-2\mu \xi^2 t)\, d \xi - \frac{1}{2\pi x} \int_{0}^{+\infty} e^{2\pi i x \xi} e^{-\mu \xi^2 t+i\beta t}(1-2\mu \xi^2 t)\, d \xi.
\end{split}
\end{equation*}  

Note that, in contrast with identity (\ref{Iden-Ker-1}), the symbol $i\xi$ (corresponding to the derivative $\partial_x$) cancels all remaining terms at $\xi = 0^{-}$ and $\xi = 0^{+}$. Consequently, we can repeat the same argument to obtain 

\begin{equation*}
\begin{split}
&\partial_x K(t,x)\\
=&\, - \frac{1}{4\pi^2 i x^2} \int_{-\infty}^{0} \partial_{\xi}(e^{2\pi i x \xi}) e^{-\mu \xi^2 t-i\beta t}(1-2\mu \xi^2 t)\, d \xi - \frac{1}{4\pi^2 i x^2} \int_{0}^{+\infty} \partial_{\xi}(e^{2\pi i x \xi}) e^{-\mu \xi^2 t+i\beta t}(1-2\mu \xi^2 t)\, d \xi\\
=&\, \frac{1}{4\pi^2 i x^2} \left(\int_{-\infty}^{0} (e^{2\pi i x \xi}) \partial_{\xi}\big( e^{-\mu \xi^2 t-i\beta t}(1-2\mu \xi^2 t)\big)\, d \xi+ \left. (e^{2\pi i x \xi})\big( e^{-\mu \xi^2 t-i\beta t}(1-2\mu \xi^2 t)\big)\right|_{-\infty}^{0}\right) \\
&\, + \frac{1}{4\pi^2 i x^2} \left(\int_{0}^{+\infty} (e^{2\pi i x \xi}) \partial_{\xi}\big( e^{-\mu \xi^2 t+i\beta t}(1-2\mu \xi^2 t)\big)\, d \xi+ \left. (e^{2\pi i x \xi})\big( e^{-\mu \xi^2 t+i\beta t}(1-2\mu \xi^2 t)\big)\right|_{0}^{+\infty}\right)\\
=&\,  -  \frac{1}{4\pi^2 i x^2} \left(\int_{-\infty}^{0} e^{2\pi i x\xi} e^{-\mu \xi^2 t- i \beta t}\,  (4\mu^2 \xi^3 t^2 -6 \mu \xi t)\, d \xi + \int_{0}^{+\infty} e^{2\pi i x\xi} e^{-\mu \xi^2 t+ i \beta t}\,  (4\mu^2 \xi^3 t^2 -6 \mu \xi t)\, d \xi \right)\\
&\, + \frac{e^{-i\beta t}-e^{i \beta t}}{4\pi^2 i x^2}\\
=:&\, - \frac{1}{4\pi^2 i x^2}  J(t,x) - \frac{\sin(\beta t)}{4\pi^2 i x^2}.
\end{split}
\end{equation*}

As before, using the rapid decay of $e^{-\mu \xi^2 t}$, we obtain 
\[ |J(t,x)| \leq \int_{-\infty}^{+\infty} e^{-\mu |\sqrt{t}\xi|^2} \Big(4 \mu^2 |\sqrt{t}\xi|^3+6\mu |\sqrt{t}\xi|\Big)\sqrt{t}d \xi \leq C<+\infty, \]
for some constant $C>0$ depending on $\mu$. Hence, for $x \neq 0$, the desired estimate (\ref{Estim-Kernel-Derivative}) follows.

\medskip

\paragraph{\bf Acknowledgements.} 
All the authors were partially supported by  MathAmSud WAFFLE 23-MATH-18. 

\medskip

\paragraph{{\bf Statements and Declaration}}
Data sharing does not apply to this article as no datasets were generated or analyzed during the current study.  In addition, the authors declare that they have no conflicts of interest, and all of them have equally contributed to this paper.

\end{document}